\documentclass[reqno]{amsart}
\usepackage{amssymb,latexsym,amsmath,amsthm,enumerate,amsbsy}
\usepackage{cite}
\usepackage[hidelinks]{hyperref}

\makeatletter
\@namedef{subjclassname@2020}{\textup{2020} Mathematics Subject Classification}
\makeatother

\theoremstyle{plain}
\newtheorem{theorem}{Theorem}[section]
\newtheorem{lemma}[theorem]{Lemma}
\newtheorem{corollary}[theorem]{Corollary}
\newtheorem{proposition}[theorem]{Proposition}
\theoremstyle{definition}
\newtheorem{remark}[theorem]{Remark}

\DeclareMathOperator{\supp}{supp}

\hypersetup{
  pdftitle={Finite bases for Brandt semirings and naturally semilattice-ordered inverse semigroups},
  pdfauthor={Xiaolei Shao and Mengya Yue}
}

\begin{document}

\title[Finite bases for Brandt semirings]
{Finite bases for Brandt semirings and naturally semilattice-ordered inverse semigroups}

\author{Xiaolei Shao}
\address{School of Mathematics, Northwest University, Xi'an, 710127, Shaanxi, P.R. China}

\author{Mengya Yue}
\address{School of Mathematics, Northwest University, Xi'an, 710127, Shaanxi, P.R. China}

\subjclass[2020]{16Y60, 03C05, 08B15, 08B05}
\keywords{additively idempotent semiring, Brandt semigroup, finite basis problem, inverse semigroup, strongly nonfinitely based}
\date{}

\begin{abstract}
We prove that Brandt semirings over nontrivial groups with at
least two indices admit no identity basis with a bounded number
of variables. A uniform family of finite witnesses has size
linear in the variable bound and applies to groups of arbitrary
cardinality. A commutative construction yields the corresponding
result for flat extensions of groups of unbounded exponent and
gives an exact finite basis criterion for arbitrary flat group
extensions. Structural representations lead to a classification
of finite natural inverse semirings and a single
identity characterizing locally Clifford natural inverse
structure among regular semirings. We also describe products,
orthogonal sums, minimum generators, colored graph identities
and linear variable complexity. Further results concern infinite
generating algebras and inherent nonfinite basability.
\end{abstract}

\maketitle

\section{Introduction}
\label{sec:introduction}

The finite basis problem asks whether the identities of an
algebra follow from a finite set of identities. For semirings,
the interaction between addition and multiplication gives rise
to finite-basis phenomena that are sensitive to both operations.
Naturally semilattice-ordered inverse semigroups provide a setting in which
the additive operation is determined by the multiplicative
structure: whenever binary meets exist in the natural order,
they form a distributive semilattice operation.

A Brandt semigroup consists of matrix units labelled by elements
of a group, together with a zero. Giving it flat addition,
so that the sum of two distinct elements is zero, produces a
Brandt semiring. The number of indices and the structure of the
group play different roles in its identities. With one index,
the construction is the flat extension of the group. With at
least two indices, it also contains the five-element semiring
of matrix units. These two constructions determine the variety
of every Brandt semiring through a join formula.

Gusev and Volkov \cite{GV} studied semiring identities of finite
inverse semigroups with natural meet addition. Together with
Shaprynski\v{\i}'s finite basis for the five-element Brandt
semiring \cite{Sh}, their results settle the finite combinatorial
case. Jackson's work \cite{Jackson08,Jackson19} gives the finite
flat-group and natural Clifford classifications. Related methods
and finite-basis questions for flat semirings appear in
\cite{JRZ,RJZL}. For a finite group with a nonabelian Sylow
subgroup, the nonfinite basability of its Brandt semiring follows
from \cite[Theorem 6.1]{JRZ}.

Our first main result determines the finite-basis problem for
Brandt semirings over arbitrary groups and arbitrary index sets
of size at least two: such a semiring is finitely based exactly
when the group is trivial. In every other case its identities
have no basis with a bounded number of variables
(Theorem~\ref{thm:all-groups} and
Corollary~\ref{cor:brandt-classification}). The proof uses one
family of finite flat semirings simultaneously for all
nontrivial groups. For each prescribed variable bound, the
witness lies outside the variety generated by all Brandt
semirings, while each subsemiring generated by sufficiently few
elements lies in every nontrivial cyclic Brandt variety. The
witnesses have linear size (Theorem~\ref{thm:uniform}).
They also give intervals of nonfinitely based varieties and
linear lower and upper bounds on variable complexity.

The construction starts from the vertex-stars of an edge-colored
prism graph. Removing one star gives a semiring that fails a
universal Brandt identity. Small generated subsemirings are
represented using integral incidence relations, group-labelled
matrix units and Rees quotients. A colored-graph formulation
describes the corresponding identities for a wider class of
graphs. A commutative version gives uniform witnesses for flat
extensions of groups of unbounded exponent. The nonfinite
basability of the flat extension of the infinite cyclic group
was established in \cite{MS25,GRZ26}; our construction yields
the bounded-variable conclusion for every group of unbounded
exponent. Together with the finite-exponent correspondence of
\cite{RJZL}, it gives the criterion in
Theorem~\ref{thm:flat-group-all}.

Subdirect Brandt representations connect these results to
natural inverse semirings. A finite such semiring
is finitely based exactly when it is locally Clifford and
combinatorial, or when it is Clifford and all Sylow subgroups
of its maximal subgroups are abelian
(Theorem~\ref{thm:finite-natural-classification}). This answers
\cite[Problem 7.7(3)]{JRZ}. Among regular semirings, a single
four-variable identity characterizes locally Clifford inverse
multiplication with natural meet addition
(Theorem~\ref{thm:regular-flat-identity}). The representations
also yield exact formulas for products and orthogonal sums,
minimum-size generators of cyclic Brandt varieties, and
classifications of the smallest negative examples. Infinite
constructions and finite generating replacements extend the
analysis of finite basability, and Zimin minimality gives the
inherent-nonfinite-basis classification.

Section~\ref{sec:preliminaries} fixes the notation and basic
constructions. Section~\ref{sec:brandt} proves the Brandt
results and their graph and quantitative forms.
Section~\ref{sec:unbounded-flat-groups} treats flat group
extensions. The structural representations and finite
classifications occupy Section~\ref{sec:natural-classification}.
Section~\ref{sec:products-infinite} concerns products, adjoining
identities and infinite generating algebras, and
Section~\ref{sec:inherent} treats inherent nonfinite basability.

\section{Preliminaries}
\label{sec:preliminaries}

\subsection{Semirings and inverse semigroups}

An \emph{additively idempotent semiring} (ai-semiring) has
associative binary operations $+$ and $\cdot$, with commutative
idempotent addition and multiplication distributive on both sides.
No multiplicative identity or additive identity is required.

A \emph{word} is a nonempty finite product of variables.
Every constant-free semiring term can be represented as a
finite nonempty sum of words.  We denote words and terms by
bold lowercase letters $\mathbf{u},\mathbf{v},\mathbf{w},\ldots$,
and use ordinary italic letters for variables and algebra elements.
For terms $\mathbf{u},\mathbf{v}$, we write
$\mathbf{u}\preceq\mathbf{v}$ for the identity
$\mathbf{u}+\mathbf{v}\approx\mathbf{v}$.
The \emph{additive order} on an ai-semiring is defined by
\[
 a\leq b\quad\Longleftrightarrow\quad a+b=b.
\]

A \emph{flat semiring} here has an element $0$ with
\begin{equation}\label{eq:flat}
 x0=0x=0,\qquad
 x+y=\begin{cases}x,&x=y,\\0,&x\ne y.\end{cases}
\end{equation}
Thus $x+0=0$; the designated zero absorbs addition and is not
an additive identity, except in the one-element algebra.
In the language $\{+,\cdot\}$, it is an element rather than a
constant symbol.  All constructions also preserve the constant
when the language is $\{+,\cdot,0\}$.

For a cancellative semigroup $S$, its \emph{flat extension}
$\flat(S)$ is obtained by adjoining an absorbing zero and
using flat addition \cite{JRZ,RJZL}.

For a group $G$ and a nonempty set $I$, define
\[
 B_I(G)=\{0\}\dot\cup
 \{(i,g,j):i,j\in I,\ g\in G\}
\]
with flat addition and multiplication
\begin{equation}\label{eq:brandt}
 (i,g,j)(k,h,\ell)=
 \begin{cases}(i,gh,\ell),&j=k,\\0,&j\ne k,\end{cases}
 \qquad x0=0x=0.
\end{equation}
This is the flat Brandt ai-semiring.
Write $B_n(G)$ when $I=\{1,\ldots,n\}$.
For the trivial group, we abbreviate $B_I(\{1\})$ and
$B_n(\{1\})$ to $B_I$ and $B_n$, respectively.
In particular, $B_2$ is the five-element Brandt semiring.
We write $e_{ij}=(i,1,j)$ for the matrix units.
For an additively written group $K$, the label of a nonzero
product is the sum of its labels; $0_K$ denotes the group
identity, distinct from the semiring zero.

Write $\mathsf{V}(A)$ for the variety generated by an algebra $A$.
We use closure under subalgebras, direct products and homomorphic
images, and the equivalence
$S\in\mathsf{V}(A)$ if and only if $S$ satisfies all identities of $A$.
Let
\[
 \mathcal{B}=\mathsf{V}\{B_n(G):n\geq2,\ G\text{ a group}\}.
\]
An algebra is \emph{finitely based} if its identities follow
from a finite family of identities. A \emph{bounded-variable
basis} is an identity basis in which a fixed finite number
bounds the number of distinct variables in each identity.
We use the specified semiring language throughout.

An inverse semigroup is a semigroup in which every element
$s$ has a unique inverse $s^{-1}$ satisfying
$ss^{-1}s=s$ and $s^{-1}ss^{-1}=s^{-1}$.
Its natural partial order $\leq_{\rm nat}$ is defined by
\[
 a\leq_{\rm nat} b\quad\Longleftrightarrow\quad a=ba^{-1}a.
\]
We call the semiring obtained by taking natural meet as addition
a \emph{naturally semilattice-ordered inverse semigroup},
or a \emph{natural inverse semiring} for short.
In this case the additive order is the reverse of
$\leq_{\rm nat}$.

Here \emph{regular} always means regularity of the
multiplicative semigroup: for every $a$ there is $x$
with $axa=a$.  A semigroup is \emph{combinatorial} if
all its subgroups are trivial.  An inverse semigroup is
\emph{Clifford} if every element lies in a subgroup.

An inverse semigroup is \emph{locally Clifford} if every
local monoid $eSe$, $e\in E(S)$, is Clifford.
Every Clifford inverse semigroup and every primitive
inverse semigroup has this property.
For general inverse semigroups we distinguish
\[
 d(s)=s^{-1}s,\qquad r(s)=ss^{-1}.
\]
Natural meets distribute over multiplication
whenever all binary meets exist.  Indeed, if $u\leq_{\rm nat} as,at$,
then $a^{-1}u\leq_{\rm nat} a^{-1}as\leq_{\rm nat} s$ and similarly
$a^{-1}u\leq_{\rm nat} t$, while $aa^{-1}u=u$.
Thus $u\leq_{\rm nat} a(s\wedge t)$.  The opposite inequality
follows from monotonicity, proving
$as\wedge at=a(s\wedge t)$; the right-hand version
is proved symmetrically.

For $\lvert I\rvert\geq2$, let $B_I^1(G)$ denote the Brandt
semigroup with a multiplicative identity adjoined, equipped
with natural meet addition. In particular, $B_2^1$ is the
six-element naturally semilattice-ordered Brandt monoid.
Both signatures treat its identity as an element of the algebra.

\subsection{Brandt constructions and absorbing zeros}
\label{sec:orthogonal}

\begin{lemma}[Zero smash products]\label{lem:smash}
For a nonempty family of groups $G_\lambda$ and nonempty
index sets $I_\lambda$, let $J$ consist of the tuples
having at least one zero coordinate in
$\prod_\lambda B_{I_\lambda}(G_\lambda)$.
Collapsing $J$ and keeping every other tuple distinct is a
semiring quotient, and
\begin{equation}\label{eq:smash}
 \left(\prod_\lambda B_{I_\lambda}(G_\lambda)\right)/J
 \cong B_{\prod_\lambda I_\lambda}
          \left(\prod_\lambda G_\lambda\right).
\end{equation}
The family and the individual sets may be infinite.
\end{lemma}

\begin{proof}
A zero coordinate stays zero after multiplication on either
side or addition of any tuple.  Hence $J$ is a two-sided
multiplicative ideal and absorbs addition; its Rees relation
is a congruence.  For a surviving tuple, collect its left
indices, group labels, and right indices into three tuples.
A product survives exactly when every pair of corresponding
middle indices agrees.  Two distinct surviving tuples differ
in some coordinate, so their sum has a zero coordinate and
collapses.  These facts prove both operations in
\eqref{eq:smash}, without a finiteness assumption.
\end{proof}

For every group $G$, the flat extension satisfies
$\flat(G)=B_1(G)$.

\begin{proposition}\label{prop:join}
For every group $G$ and nonempty set $I$,
\begin{equation}\label{eq:join}
 \mathsf{V}(B_I(G))=\mathsf{V}(B_I)\vee\mathsf{V}(\flat(G)).
\end{equation}
If $|I|\geq2$, then
\[
 \mathsf{V}(B_I)=\mathsf{V}(B_2),\qquad
 \mathsf{V}(B_I(G))=\mathsf{V}(B_2(G)).
\]
\end{proposition}

\begin{proof}
The two-factor zero smash identifies $B_I(G)$ as a quotient
of $B_I\times \flat(G)$.  Both factors embed
in $B_I(G)$, using identity group labels or one diagonal
corner, proving \eqref{eq:join}.
For finite $I$ with $|I|\geq2$, choose $r$ with
$2^r\geq|I|$.  The zero smash of $r$ copies of
$B_2$ is $B_{2^r}$, which contains
$B_I$.  The reverse inclusion uses two indices.
For infinite $I$, each evaluation of an identity uses only
finitely many indices and lies in a finite-index Brandt
subsemiring.  Thus the same identities hold.  Equation
\eqref{eq:join} gives the assertion for arbitrary $G$.
\end{proof}

For a nonempty family of flat semirings, their \emph{orthogonal
sum} identifies their zeros and sets every mixed product and
mixed sum of nonzero elements from different summands equal
to zero.  It embeds as the union of the coordinate axes in
their direct product, and contains every summand.

\begin{theorem}[Exact variety of an orthogonal sum]
\label{thm:orthogonal}
Let
\[
 S=\bigoplus_{\lambda\in\Lambda}^{0} B_{I_\lambda}(G_\lambda),
 \qquad G=\prod_{\lambda\in\Lambda}G_\lambda,
\]
where $\Lambda$ and every $I_\lambda$ are nonempty.
The groups and all indexing sets may be infinite.  Then
\begin{equation}\label{eq:orthogonal-normal-form}
 \mathsf{V}(S)=
 \begin{cases}
  \mathsf{V}(\flat(G)),&|I_\lambda|=1\text{ for every }\lambda,\\
  \mathsf{V}(B_2(G)),&|I_\lambda|\geq2\text{ for some }\lambda.
 \end{cases}
\end{equation}
\end{theorem}

\begin{proof}
The coordinate-axis embedding and the inclusions of the
summands give
\[
 \mathsf{V}(S)=\bigvee_\lambda\mathsf{V}(B_{I_\lambda}(G_\lambda)).
\]
Similarly, by the arbitrary-product zero-smash lemma and
the coordinate subgroup embeddings,
\[
 \bigvee_\lambda\mathsf{V}(\flat(G_\lambda))
   =\mathsf{V}\bigl(\flat(\prod_\lambda G_\lambda)\bigr).
\]
If all index sets are singletons, this is the assertion.
Otherwise $S$ also contains $B_2$, and
Proposition~\ref{prop:join} gives
$B_2(G)\in\mathsf{V}(S)$.
Conversely, every group $G_\lambda$ embeds in
$G$, so every summand belongs to
$\mathsf{V}(B_2(G))$, by index independence.  The
coordinate-axis embedding gives the reverse inclusion.
\end{proof}

In particular, the matrix-unit component and the nontrivial
group need not occur in the same summand.  If some index set
has size at least two and some $G_\lambda\ne\{1\}$,
Theorem~\ref{thm:all-groups} applies to the entire sum.
The empty orthogonal sum is the zero-only algebra and is
finitely based separately.

\begin{lemma}[Conservativity of an absorbing constant]
\label{lem:constant}
Let $S$ be any ai-semiring with an element $0$
absorbing both operations.  It is finitely based in
$\{+,\cdot\}$ if and only if its expansion naming
that zero is finitely based in $\{+,\cdot,0\}$.
The existence of a basis with a uniform finite variable
bound is also equivalent in the two signatures.
No nonzero idempotent hypothesis is needed.
\end{lemma}

\begin{proof}
Given a binary basis, add
\[
 x+0\approx0,\qquad x0\approx0,\qquad0x\approx0.
\]
Every term containing zero reduces to zero.
If a valid expanded identity has the form $\mathbf{p}\approx0$
with $\mathbf{p}$ constant-free, then the binary identity
$\mathbf{p}+y\approx \mathbf{p}$, with $y$ fresh, is valid in $S$.
It follows from the binary basis, and substituting $0$
for $y$ proves $\mathbf{p}\approx0$.
The remaining identities are already binary or reduce
to $0\approx0$.  This proves the forward implication.

Conversely, let $\Sigma$ be a basis of the expansion;
include the ai-semiring axioms if working relative to
that class.  First suppose a constant-free term is
identically zero in $S$.  Identifying all its variables
gives a unary term $\mathbf{q}(x)$ with that property.
In each member of $\Sigma$, replace every $0$ by
$\mathbf{q}(z)$, with $z$ a fresh variable, and add
\[
 \mathbf{q}(x)\approx \mathbf{q}(y),\qquad \mathbf{q}(x)+y\approx \mathbf{q}(x),\qquad
 \mathbf{q}(x)y\approx \mathbf{q}(x),\qquad y\mathbf{q}(x)\approx \mathbf{q}(x).
\]
All these identities are valid in the binary reduct of
$S$.  In any model they make the value of $\mathbf{q}$
a well-defined common absorbing zero.  Naming this
value gives a model of $\Sigma$, so the binary
model satisfies every identity of $S$.
Thus the displayed construction gives a binary basis.

Now suppose there is no constant-free zero term.
Every valid binary identity of $S$ has the same
variable content on both sides: otherwise an exclusively
occurring variable can be assigned zero, and all other
variables a single variable value $x$, producing a
constant-free unary zero term on the opposite side.
For the same reason a valid expanded identity either
contains no zero constant or contains zero on both sides.
Let $\Sigma_{\rm free}$ be the constant-free members
of $\Sigma$, and let $A$ be any binary model of
$\Sigma_{\rm free}$.  Adjoin a fresh element
$\infty$ absorbing both operations.
Matching variable content ensures that every identity
of $\Sigma_{\rm free}$ remains valid: either all
occurring values lie in $A$, or both sides become
$\infty$.  Name $\infty$ as zero.
Every remaining member of $\Sigma$ is then automatic,
so the expansion satisfies $\Sigma$, and its binary
subalgebra $A$ satisfies all binary identities of $S$.
This proves completeness of $\Sigma_{\rm free}$.

The constructions preserve finiteness.  For a possibly
infinite basis with a uniform variable bound they add
at most one variable, and only a fixed finite family
of auxiliary identities.  This proves the last assertion.
\end{proof}

\section{Brandt semirings}
\label{sec:brandt}

The following results determine the finite-basis problem for
Brandt semirings with at least two indices.

\begin{theorem}\label{thm:all-groups}
For every nontrivial group $G$ and every set $I$ with
$|I|\geq2$, the flat Brandt semiring $B_I(G)$ has no
identity basis with a uniform finite bound on the number of
variables.  In particular it is not finitely based.
This holds in both $\{+,\cdot\}$ and $\{+,\cdot,0\}$.
\end{theorem}

We prove the following uniform finite-witness statement.

\begin{theorem}\label{thm:uniform}
For every $k\geq1$, put
\[
 m=\max\{2,\lfloor k/4\rfloor+1\}.
\]
There is a finite flat semiring $W_k$ with exactly
$10m+2$ elements such that:
\begin{enumerate}[$(\rm i)$]
\item $W_k\notin\mathcal{B}$;
\item every subsemiring of $W_k$ generated by at most $k$
elements belongs to $\mathsf{V}(B_2(C_r))$ for every integer $r\geq2$,
and also belongs to $\mathsf{V}(B_2(\mathbb{Z}))$.
\end{enumerate}
Consequently the same $W_k$ satisfies every identity in at most
$k$ variables of every $B_I(G)$ with $|I|\geq2$ and
$G\ne\{1\}$.
\end{theorem}

\begin{corollary}[An interval of nonfinitely based varieties]\label{cor:interval}
If a variety $\mathcal W\subseteq\mathcal{B}$ contains $B_2(C_p)$
for some prime $p$, or contains $B_2(\mathbb{Z})$, then
$\mathcal W$ has no finite identity basis.  In fact its
identities have no basis with a uniform finite bound on the
number of variables.
\end{corollary}

\begin{corollary}[All Brandt index sets]\label{cor:brandt-classification}
For every group $G$ and every set $I$ with $|I|\geq2$,
in either stated language,
\[
 B_I(G)\text{ is finitely based}\quad\Longleftrightarrow\quad G=\{1\}.
\]
\end{corollary}

\begin{proof}
The negative direction is Theorem~\ref{thm:all-groups}.
For the positive direction, index independence reduces to
$B_2$, whose finite basis in the binary language
is given by \cite[Theorem 1]{Sh}.
Lemma~\ref{lem:constant} deals with the named zero.
\end{proof}

\subsection{Prism semirings and identities}

A three-partite three-uniform hypergraph has disjoint vertex
parts $A,B,C$, and every edge is a triple $abc$, with
$a\in A,b\in B,c\in C$.
It is \emph{linear} if two distinct edges meet in at most one
vertex, equivalently a specified pair of vertices belongs to
at most one edge.

Let $H$ be a finite nonempty linear hypergraph of this kind,
with no isolated vertices.
On the disjoint union
\begin{equation}\label{eq:SH-carrier}
 S_H=A\dot\cup B\dot\cup C
 \dot\cup\{p_a:a\in A\}
 \dot\cup\{q_c:c\in C\}
 \dot\cup\{h,0\},
\end{equation}
where the $p_a,q_c,h$ are new formal symbols, use flat addition
and the nonzero products
\begin{equation}\label{eq:SH-products}
 ab=q_c,\quad bc=p_a\quad(abc\in H),\qquad
 ap_a=h,\quad q_cc=h.
\end{equation}
All other products are zero.

\begin{lemma}\label{lem:SH}
The operations above make $S_H$ a flat ai-semiring.
Its only nonzero products of three factors are $abc=h$ for
edges $abc\in H$.
Every product of four factors, and every square, is zero.
\end{lemma}

\begin{proof}
Linearity makes the multiplication table unambiguous.
A nonzero $(xy)z$ must have intermediate product $q_c$;
hence $x=a,y=b,z=c$ for an edge $abc$, and its value is $h$.
A nonzero $x(yz)$ must have intermediate product $p_a$,
giving exactly the same condition and value.
Conversely an edge gives
$a(bc)=ap_a=h=q_cc=(ab)c$.
This proves associativity and the description of triple products.
The other vanishing assertions follow from the disjoint types.

The multiplication is $0$-cancellative:
\[
 xy=xz\ne0\ \Longrightarrow\ y=z,\qquad
 yx=zx\ne0\ \Longrightarrow\ y=z.
\]
For instance, with fixed left factor $a$, two products with
$B$-elements having the same output $q_c$ would give two
edges sharing $a,c$, contrary to linearity.
The output $h$ occurs only with right factor $p_a$.
With fixed left factor $b$, equal nonzero outputs $p_a$
would give two edges sharing $a,b$; a left factor $q_c$
has only one nonzero right multiple.
These exhaust the cases.
The corresponding argument for a fixed right factor uses the pairs
$b,c$ and $a,c$ for edge products.

$0$-cancellation implies distributivity
under flat addition.
If $y\ne z$, the products $xy,xz$ are either different or
both zero, since they cannot equal the same nonzero element.
Thus $xy+xz=0=x(y+z)$.
The case $y=z$ is immediate; right distributivity is proved
in the same way.
\end{proof}

Fix $m\geq2$, with indices read modulo $m$.
Use the distinct formal letters
\[
 A=\{a_i,\bar a_i:0\leq i<m\},\quad
 B=\{b_i,d_i:0\leq i<m\},\quad
 C=\{c_i,\bar c_i:0\leq i<m\}.
\]
Bars are part of the names, not inverse operations.
The full hypergraph has the following $4m$ edges:
\begin{equation}\label{eq:prism-faces}
 \begin{aligned}
 U_i&=a_i b_i c_i,&
 V_i&=a_{i+1}d_i c_i,\\
 \bar U_i&=\bar a_i b_i\bar c_i,&
 \bar V_i&=\bar a_{i+1}d_i\bar c_i.
 \end{aligned}
\end{equation}
Let $H_m$ consist of all these edges except $U_0$.

Each letter in the full hypergraph occurs in exactly two edges.
Make a graph $\Gamma_m$ whose vertices are the $4m$ edges in
\eqref{eq:prism-faces}, and whose graph edges are their shared
letters.  Its upper and lower cycles are
\[
 U_0,V_0,U_1,V_1,\ldots,U_{m-1},V_{m-1},U_0
\]
and the corresponding barred cycle.
The rungs join $U_i$ to $\bar U_i$ through $b_i$, and
$V_i$ to $\bar V_i$ through $d_i$.
Thus
\begin{equation}\label{eq:prism-graph}
 \Gamma_m\cong C_{2m}\mathbin{\square}K_2.
\end{equation}
It is connected, simple, cubic and bipartite.
The sets $\{U_i,\bar V_i\}$ and $\{V_i,\bar U_i\}$
give its bipartition.
The original hyperedges are exactly its vertex-stars,
with graph edges coloured by $A,B,C$.

In particular the hypergraph is linear.
After $U_0$ is removed, the letters $a_0,b_0,c_0$ still
occur in $V_{m-1},\bar U_0,V_0$, respectively.
Hence $H_m$ has no isolated vertices, all its vertex degrees
are at most two, and it has $4m-1$ edges.
The semiring $S_{H_m}$ has
\begin{equation}\label{eq:witness-size}
 |S_{H_m}|=2m+2m+2m+2m+2m+2=10m+2.
\end{equation}

Use the letters as distinct variables and define
\[
 \mathbf{u}_m=\sum_{\mathbf{w}\in H_m}\mathbf{w},\qquad \mathbf{w}_0=a_0b_0c_0.
\]

\begin{lemma}[The prism identity]\label{lem:prism-identity}
For every group $G$ and every $n\geq1$,
\begin{equation}\label{eq:prism-identity}
 B_n(G)\models \mathbf{w}_0\preceq\mathbf{u}_m .
\end{equation}
This identity fails in $S_{H_m}$.
\end{lemma}

\begin{proof}
Consider any evaluation in $B_n(G)$.
If $\mathbf{u}_m=0$, the identity holds because zero absorbs addition.
Otherwise flat addition implies that every retained word in
\eqref{eq:prism-faces} has the same nonzero value
$(s,\delta,t)$.
All variables are nonzero, since there are no isolated
vertices in $H_m$.
Write the value of a variable $x$ as
$(\ell_x,g_x,r_x)$.

The matrix indices of the missing word follow from the retained words.
The retained words
$a_0d_{m-1}c_{m-1}$,
$\bar a_0d_{m-1}\bar c_{m-1}$, and
$\bar a_0b_0\bar c_0$ give
\[
 r_{a_0}=\ell_{d_{m-1}}=r_{\bar a_0}=\ell_{b_0}.
\]
The retained words
$\bar a_0b_0\bar c_0$,
$\bar a_1d_0\bar c_0$, and
$a_1d_0c_0$ give
\[
 r_{b_0}=\ell_{\bar c_0}=r_{d_0}=\ell_{c_0}.
\]
Also $\ell_{a_0}=s$ and $r_{c_0}=t$, using
$V_{m-1}$ and $V_0$.
Thus $\mathbf{w}_0$ is composable and has the required outside indices.

We next determine the group label.
In the next calculation use the same letter for the group label
of its evaluated variable, to keep the notation short.
All lower words are retained, so right cancellation gives
\begin{equation}\label{eq:lower-relations}
 \bar a_i b_i=\bar a_{i+1}d_i
 \qquad\text{for every }i.
\end{equation}
For $i\ne0$, the two upper words are also retained, and
\begin{equation}\label{eq:upper-relations}
 a_i b_i=a_{i+1}d_i.
\end{equation}
Put $r_i=a_i\bar a_i^{-1}$.
Equation \eqref{eq:lower-relations} gives
$b_i d_i^{-1}=\bar a_i^{-1}\bar a_{i+1}$.
Together with \eqref{eq:upper-relations}, this yields
$r_i=r_{i+1}$ for $i=1,\ldots,m-1$.
Since indices are modulo $m$, these equalities imply
$r_1=r_2=\cdots=r_{m-1}=r_0$.
Now \eqref{eq:lower-relations} at $i=0$ gives
\[
 a_0b_0
 =r_0\bar a_0b_0
 =r_1\bar a_1d_0
 =a_1d_0.
\]
Multiplying by $c_0$, the retained word $V_0$ shows
$a_0b_0c_0=\delta$.
The missing word therefore has value $(s,\delta,t)$,
proving \eqref{eq:prism-identity}.

In $S_{H_m}$, assign every variable its identically named
letter.  All retained words have value $h$ by
Lemma~\ref{lem:SH}, and the missing word has value zero.
Thus the inequality in \eqref{eq:prism-identity} fails,
since $0\nleq h$ in the additive order.
\end{proof}

Consequently $S_{H_m}\notin\mathcal{B}$ for every $m\geq2$.
The rest of the proof shows that subsemirings generated by few
elements nevertheless belong to very small Brandt varieties.

\subsection{Incidence relations and local representations}

For a graph $\Gamma$ and a vertex set $X$, write
$\partial_\Gamma X$ for the set of graph edges with precisely
one endpoint in $X$.

\begin{lemma}[Small cuts of a prism]\label{lem:prism-cuts}
If $\Gamma=C_r\square K_2$, $r\geq4$, and
$\varnothing\ne X\ne V(\Gamma)$, then
$|\partial_\Gamma X|\geq3$.
Equality holds only when $X$ or its complement is a singleton.
\end{lemma}

\begin{proof}
Describe $X$ by its upper and lower index sets $I,J$
in the cycle $C_r$.
Then
\begin{equation}\label{eq:cut-formula}
 |\partial_\Gamma X|
 =|\partial_{C_r}I|+|\partial_{C_r}J|+|I\mathbin{\triangle}J|.
\end{equation}
If both $I,J$ are nonempty proper subsets of the cycle,
each of the first two terms is at least two.
Otherwise, after exchanging layers or taking the complement
of $X$, assume $I=\varnothing$.
The set $J$ is nonempty.
If $J$ is the full cycle, the cut has size $r\geq4$.
Otherwise the size is
$|\partial_{C_r}J|+|J|\geq2+|J|$.
It is at most three precisely when $J$ is a singleton.
Undoing the complement and layer exchange gives the assertion.
\end{proof}

For the local arguments, let $\Gamma$ be any finite connected
simple cubic bipartite graph with a proper edge coloring by
$A,B,C$, satisfying the following cut condition:
\begin{equation}\label{eq:cut-condition}
 |\partial_\Gamma X|\geq3\quad
 (\varnothing\ne X\ne V(\Gamma)),\qquad
 |\partial_\Gamma X|=3\Longrightarrow
 \min\{|X|,|V(\Gamma)\setminus X|\}=1.
\end{equation}
The prisms with their specified coloring satisfy these hypotheses.
Let $F$ be a nonempty proper subset of $V(\Gamma)$,
also regarded as the hypergraph consisting of its vertex-stars.
Its hypergraph vertex set is
\[
 E_F=\{e\in E(\Gamma):e\text{ has an endpoint in }F\}.
\]
This hypergraph has no isolated vertices.
Take any nonzero commutative ring $R$ with identity,
and use the standard basis $e_x$, $x\in E_F$, of
$R^{E_F}$.
For $v\in F$, put
\begin{equation}\label{eq:relation-module}
 \rho_v=\sum_{x\ni v}e_x,\qquad
 W_F=\operatorname{span}_{R}\{\rho_v:v\in F\}.
\end{equation}

\begin{lemma}[Short-support relations]\label{lem:short-support}
Suppose that at least two vertices of $\Gamma$ lie outside
$F$.
The vectors $\rho_v$, $v\in F$, are independent over
$R$.
Every nonzero vector of $W_F$ supported on at most three
coordinates is a nonzero scalar multiple of a single
$\rho_v$, $v\in F$.
In particular $W_F$ contains no nonzero vector supported on
at most two coordinates, and
\begin{equation}\label{eq:triple-test}
 e_a+e_b+e_c\in W_F
 \quad\Longleftrightarrow\quad abc\in F
 \qquad(a\in A\cap E_F,\ b\in B\cap E_F,\ c\in C\cap E_F).
\end{equation}
\end{lemma}

\begin{proof}
Give the bipartition classes of $\Gamma$ signs $+1,-1$,
denoted by $\varepsilon(v)$.
For a combination $z=\sum_{v\in F}\lambda_v\rho_v$, put
\[
 f(v)=
 \begin{cases}\varepsilon(v)\lambda_v,&v\in F,\\
 0,&v\notin F.\end{cases}
\]
On a graph edge $uv$, oriented from the positive to the
negative bipartition class, its coefficient in $z$ is
\[
 \lambda_u+\lambda_v=f(u)-f(v),
\]
where coefficients outside $F$ are read as zero.
Thus $\supp z$ is exactly the set of edges joining vertices
with different $f$-values.
If $z=0$, connectedness makes $f$ constant; since $F$
is proper it has value zero outside $F$, so all coefficients
are zero.  This proves independence.

Now suppose $z\ne0$ and $|\supp z|\leq3$.
Choose a nonzero value $\alpha$ of $f$, and set
$X=f^{-1}(\alpha)$.
This is a nonempty proper set, and
$\partial_{\Gamma}X\subseteq\supp z$.
By \eqref{eq:cut-condition}, $X$ or its complement is a
singleton.
The latter is impossible: at least two vertices outside $F$
have value zero, so both lie outside $X$.
Hence $X=\{v\}$, and its three incident edges constitute
all of $\supp z$.

If another vertex had a nonzero value, its nonzero level set
would similarly be a singleton $\{w\}$, with
$\partial\{w\}=\supp z=\partial\{v\}$.
Distinct vertices of a simple cubic graph cannot have identical
sets of three incident edges: that would require three parallel
edges between them.
Thus $f$ is supported only at $v$, and $z=\lambda_v\rho_v$.

Each $\rho_v$ has support exactly three.
This excludes nonzero support sizes one and two.
For \eqref{eq:triple-test}, the forward implication makes
$e_a+e_b+e_c$ a scalar multiple of one $\rho_v$.
Its support and coefficients force the star $v$ to be $abc$
and the scalar to be $1$.
The reverse implication follows from the definition.
\end{proof}

\begin{lemma}[A free quotient over every coefficient ring]\label{lem:primitive}
For any nonempty proper $F\subset V(\Gamma)$,
\[
 R^{E_F}/W_F\cong R^d,
 \qquad d=|E_F|-|F|\geq2.
\]
In particular the integer row module is a direct summand.
\end{lemma}

\begin{proof}
Multiply each row by its bipartition sign.  The result is a
restricted oriented incidence matrix, up to column signs.
Root a spanning tree of $\Gamma$ at a vertex outside $F$.
For each $v\in F$, select the column of its edge to its
parent.  These columns are distinct and belong to $E_F$.
Order the rows and their corresponding columns by increasing
depth.  A selected column has its diagonal entry $\pm1$
and at most one other entry, in an earlier row.  Thus the
square selected matrix is triangular with determinant $\pm1$.
It is invertible over every such $R$.

Let $\pi$ be projection to the selected coordinates.
Its restriction $\pi|_{W_F}$ is an isomorphism onto
$R^F$.  Subtracting the unique row combination with
prescribed selected coordinates gives
\[
 R^{E_F}=W_F\oplus\ker\pi,
 \qquad \ker\pi\cong R^{E_F\setminus J},
\]
where $J$ is the selected column set.  This proves freeness
and the rank formula, without division by nonunits or a
finite search for minors.

Let $t=|F|$, let $e$ count the graph edges internal to $F$,
and put $b=|\partial_{\Gamma}F|$.
Cubicity gives $3t=2e+b$, while $|E_F|=e+b$.
Thus
\begin{equation}\label{eq:module-rank}
 d=|E_F|-|F|=\frac{t+b}{2}\geq2,
\end{equation}
because $t\geq1$ and \eqref{eq:cut-condition} gives
$b\geq3$.
\end{proof}

\begin{lemma}[Local representation]\label{lem:local-representation}
Let $\Gamma$ satisfy the hypotheses of
\eqref{eq:cut-condition}, and let $F$ be a nonempty set of
its vertex-stars omitting at least two graph vertices.
Then
\begin{equation}\label{eq:local-varieties}
 S_F\in\mathsf{V}(B_2(C_r))\quad\text{for every integer }r\geq2,
 \qquad
 S_F\in\mathsf{V}(B_2(\mathbb{Z})).
\end{equation}
\end{lemma}

\begin{proof}
Take $R=\mathbb{Z}/r\mathbb{Z}$, $r\geq2$, or $\mathbb{Z}$, and form the additive group
\[
 K=R^{E_F}/W_F.
\]
By Lemma~\ref{lem:primitive}, it is $C_r^d$ or $\mathbb{Z}^d$,
respectively, for $d\geq2$.
Write $[x]$ for the basis class of a hypergraph vertex $x$.
This notation distinguishes quotient classes from the formal
barred letters in \eqref{eq:prism-faces}.
Lemma~\ref{lem:short-support} says that distinct vertices
have distinct classes and
\begin{equation}\label{eq:labels}
 [a]+[b]+[c]=0_K
 \quad\Longleftrightarrow\quad abc\in F.
\end{equation}

In $B_4(K)$, generate a subsemiring $T$ by
\begin{equation}\label{eq:local-gens}
 A_a=(1,[a],2),\qquad B_b=(2,[b],3),\qquad
 C_c=(3,[c],4).
\end{equation}
It contains zero, since $A_a^2=0$.
Every nonzero element of $T$ is the value of
a nonempty multiplicative word in these generators:
flat addition can only retain an element or give zero.
The increasing matrix indices restrict such words to length
at most three and positions
\begin{equation}\label{eq:positions}
 (1,2),\ (2,3),\ (3,4),\ (1,3),\ (2,4),\ (1,4).
\end{equation}
The elements
\begin{equation}\label{eq:PQH}
 H_0=(1,0_K,4),\qquad
 P_a=(2,-[a],4),\qquad Q_c=(1,-[c],3)
\end{equation}
belong to $T$.
Indeed, $F$ has an edge, every vertex of $F$ is incident
with one, and \eqref{eq:labels} gives the required two-letter
and three-letter products.
In particular $H_0$ is a nonzero Brandt element, despite its
identity group label.

Let $T^1$ be $T$ with a formal identity
adjoined for the following divisibility criterion:
\[
 D=\{x\in T:uxv=H_0
        \text{ for some }u,v\in T^1\},\qquad
 I=T\setminus D.
\]
Then $0\in I$, $H_0\notin I$, and $I$ is a two-sided
ideal: if $sx$ or $xs$ divides $H_0$, associativity
makes $x$ divide $H_0$.
Also $I+T\subseteq I$, since adding to an
element of $I$ retains it or gives zero.
Collapsing exactly $I$ is therefore a semiring Rees quotient.

All elements in \eqref{eq:local-gens} and \eqref{eq:PQH}
divide $H_0$, using incident edges.
There are no other surviving elements.
The first three positions in \eqref{eq:positions} contain
only the original generators.
An element in position $(1,3)$ is
$(1,[a]+[b],3)=A_aB_b$.
No nonempty left extension is possible, because no nonzero
element of $T$ ends at index $1$.
Its only possible extension to $H_0$ is multiplication on
the right by some $C_c$.
By \eqref{eq:labels}, this happens precisely for an edge
$abc$, when the element is $Q_c$.
Symmetrically the surviving elements in position $(2,4)$
are exactly the $P_a$.
An element in position $(1,4)$ has no nonempty extension
on either side, so it survives only if it equals $H_0$.

Matrix positions separate all the types.
Within each type distinctness follows from the absence of
a nonzero support-two relation in $W_F$.
The same-symbol bijection from $S_F$ onto $T/I$
has precisely the required products:
$A_aB_b$ and $B_bC_c$ survive exactly for edges, while
\[
 A_aP_{a'}=(1,[a]-[a'],4)
\]
survives exactly when $a=a'$, and $Q_cC_{c'}$ survives
exactly when $c=c'$.
All remaining products are zero by their positions.
Both additions are flat and zero is preserved.  Thus
\[
 S_F\cong T/I.
\]

Finally take $L=C_r$ or $L=\mathbb{Z}$, according to the
coefficient ring.
Lemma~\ref{lem:smash} gives
$B_{2^d}(L^d)\in\mathsf{V}(B_2(L))$.
As $d\geq2$ and $K\cong L^d$, selecting four indices
gives an embedded $B_4(K)$.
Its subsemiring and quotient just constructed establish
\eqref{eq:local-varieties}.
All maps used here preserve the absorbing zero.
\end{proof}

\begin{corollary}[Scalar integer labels]\label{cor:integer-labels}
Every local $S_F$ in Lemma~\ref{lem:local-representation}
is a Rees quotient of a subsemiring of $B_4(\mathbb{Z})$.
\end{corollary}

\begin{proof}
Use $R=\mathbb Q$ in Lemma~\ref{lem:short-support}.
In the quotient vector space, each difference $[x]-[y]$,
$x\ne y$, and each nonedge sum $[a]+[b]+[c]$ is nonzero.
There are finitely many of these vectors.  Choose a basis and
consider the rational functional
\[
 \ell_t(z_1,\ldots,z_d)=z_1+t z_2+\cdots+t^{d-1}z_d.
\]
For each forbidden vector this is a nonzero polynomial in
$t$, so only finitely many integer values of $t$ are
excluded altogether.  Choose another integer and clear all
denominators.  The resulting integer labels $\lambda_x$
are pairwise distinct and satisfy
$
\lambda_a+\lambda_b+\lambda_c=0\iff abc\in F.
$
The exact Rees construction in the preceding proof now works
in $B_4(\mathbb{Z})$ itself.  The choice uses only the elementary
fact that a nonzero polynomial has finitely many roots.
\end{proof}

\subsection{Finite witnesses and variable complexity}
\label{sec:capture}
\label{sec:brandt-proof}
\label{sec:interval}
\label{sec:linear-variable-complexity}

Let $H$ be obtained by deleting one star of a properly
three-colored simple cubic graph.  For a nonempty generating
set $X\subseteq S_H$, let $Z$ be the set of indices of
its elements: letters index themselves, $p_a$ has index
$a$, $q_c$ has index $c$, and $h,0$ have no index.
Let $L=X\cap(A\cup B\cup C)$ be the actual letter
generators.  Then $L\subseteq Z$ and $|Z|\leq|X|$.
Call a star $abc\in H$ \emph{mandatory} if
\[
 \{a,b\}\subseteq L\quad\text{or}\quad\{b,c\}\subseteq L,
\]
and denote the set of mandatory stars by $M$.

\begin{lemma}[Capture criterion]\label{lem:capture-criterion}
If $F\subseteq H$ is nonempty, $M\subseteq F$, and
$Z\subseteq E_F$, the same-symbol map embeds
$\langle X\rangle$ into $S_F$.
Moreover some such $F$ has at most $|X|$ stars.
\end{lemma}

\begin{proof}
No operation creates a new letter: a product has type
$p,q,h,0$, and a flat sum retains a summand or gives zero.
Every new $p$- or $q$-element therefore comes from a
product $bc$ or $ab$ of initial letters, hence from a
mandatory star.  The assumptions on $F$ supply every
initial symbol and every newly created $p,q$ index.

The resulting same-symbol map is injective and preserves
zero, which already belongs to every nonempty generated
subsemiring because squares vanish.  Each nonzero edge
product is preserved by $M\subseteq F$.  Products
$ap_a=h$ and $q_cc=h$ depend only on matching symbols
and are also preserved.  A zero product cannot become a
new edge product, since $F\subseteq H$; none of the
symbol-based products changes either.  This proves
multiplicativity, and injectivity with preservation of
zero proves preservation of flat addition.

For the quantitative assertion put $L_M=L\cap E_M$.
Every mandatory star has at least two incidences with
$L_M$, whereas each letter has degree at most two.  Hence
\[
 2|M|\leq\sum_{x\in L_M}\deg_M(x)\leq2|L_M|.
\]
For every $z\in Z\setminus E_M$, choose one incident
retained star.  Together with $M$, these form a set $F$
with
\[
 |F|\leq|M|+|Z\setminus E_M|
 \leq|L_M|+|Z|-|L_M|\leq|X|.
\]
If no star was selected, choose one; $|X|\geq1$ makes the
same bound valid.  All required inclusions hold.
\end{proof}

For prisms, the coloring gives a stronger bound.

\begin{lemma}[A second star can be omitted]\label{lem:small-generators}
For $m\geq2$, every subsemiring of $S_{H_m}$ generated
by at most $4m-1$ elements embeds into some $S_F$ with
\[
 F=V(\Gamma_m)\setminus\{v_0,w\},\qquad v_0=U_0,\quad w\ne v_0.
\]
In particular the local representation theorem applies.
\end{lemma}

\begin{proof}
Use the preceding notation for a generating set $X$, with
$|X|\leq4m-1$.  A candidate $w\in H_m$ is permitted if
it is not mandatory and the graph edge $v_0w$, if it
exists, does not belong to $Z$.
For such $w$, every edge in $Z$ has an endpoint in
$F=V(\Gamma_m)\setminus\{v_0,w\}$: an edge lacking one
would have to be $v_0w$.  Also $M\subseteq F$, so
Lemma~\ref{lem:capture-criterion} gives the embedding.

Suppose, to obtain a contradiction, that no candidate is
permitted.  Every retained vertex is then either mandatory
or adjacent to $v_0$ by an edge in $Z$.
Each rung other than $b_0$ has a retained endpoint in the
same bipartition class as $v_0$.  That endpoint is not
adjacent to $v_0$, hence is mandatory.  Its rung is
therefore an actual $B$-letter generator.  At the other
endpoint $\bar U_0$ of $b_0$, either being mandatory or
being forbidden by the edge $b_0\in Z$ also makes $b_0$
a $B$-letter generator: no $p$- or $q$-symbol has a
$B$-index.  Thus $Z$ contains all $2m$ rungs.

Every vertex of the upper cycle except $v_0$ is incident
with a selected $A$- or $C$-index.  For a mandatory
vertex this follows from its definition; otherwise its
forbidden status comes from a selected horizontal edge at
$v_0$.  Every vertex of the lower cycle except
$\bar U_0$ is mandatory and has the same property.
Covering $2m-1$ vertices of a cycle by its edges requires
at least $\lceil(2m-1)/2\rceil=m$ edges.  The two
horizontal cycles are disjoint, so $Z$ contains at least
$2m$ additional indices.  Consequently
\[
 |X|\geq|Z|\geq2m+2m=4m,
\]
contrary to the assumed bound.  A permitted $w$ exists.
For the named-zero language the empty-generated algebra is
$\{0\}$, which needs no embedding construction.
\end{proof}

\begin{remark}
This threshold is exact for this deletion scheme, with
embeddings preserving the symbols.  With $X=A\cup B$, there
are $4m$ generators and every retained star supplies a
nonzero product $ab$.  No such embedding into an $S_F$
with $F\subseteq H_m$ can delete another star.
\end{remark}

\begin{proof}[Proof of Theorem~\ref{thm:uniform}]
Take $m=\max\{2,\lfloor k/4\rfloor+1\}$ and
$W_k=S_{H_m}$.  Then $k\leq4m-1$, and
\eqref{eq:witness-size} gives $|W_k|=10m+2$.
Lemma~\ref{lem:prism-identity} supplies an identity valid
in every generator of $\mathcal{B}$, but false in $W_k$.
Hence $W_k\notin\mathcal{B}$.
Every at-most-$k$-generated subsemiring embeds, by
Lemma~\ref{lem:small-generators}, into a two-star-deleted
$S_F$.  Lemma~\ref{lem:local-representation} puts it in
$\mathsf{V}(B_2(C_r))$ for every $r\geq2$, and in
$\mathsf{V}(B_2(\mathbb{Z}))$.

For any nontrivial group $G$, choose $g\ne1$.  Its
cyclic subgroup is either $C_r$ for some $r\geq2$, or
$\mathbb{Z}$.  Choosing two indices of $I$ embeds the
corresponding $B_2(\langle g\rangle)$ into $B_I(G)$.
Thus every at-most-$k$-generated subsemiring of $W_k$
belongs to $\mathsf{V}(B_I(G))$.
An evaluation of an identity with at most $k$ variables
lies in the subsemiring generated by its assigned values;
therefore every such identity of $B_I(G)$ holds in
$W_k$.  All embeddings and quotients preserve the
absorbing zero, so the argument includes that constant.
\end{proof}

\begin{proof}[Proof of Theorem~\ref{thm:all-groups}]
Suppose an identity basis $\Sigma$, possibly infinite,
used at most $k$ variables in each identity.
Theorem~\ref{thm:uniform} gives $W_k\models\Sigma$.
But the corresponding prism identity holds in
$B_I(G)$ and fails in $W_k$.  Its validity for an
arbitrary index set follows from exactly the same two
index chains and group cancellation calculation in
Lemma~\ref{lem:prism-identity}.  Thus $\Sigma$ cannot
be a basis, a contradiction.
\end{proof}

\begin{proof}[Proof of Corollary~\ref{cor:interval}]
Choose the local representation over the stated prime
cyclic group or over $\mathbb{Z}$.  Every small generated
subsemiring of $W_k$ then belongs to $\mathcal W$,
whereas $W_k\notin\mathcal{B}$ implies
$W_k\notin\mathcal W$.  The same argument rules out
any bounded-variable basis.
\end{proof}

\begin{remark}
For $k\geq4$, the witness bound is
$|W_k|\leq\frac52k+12$.  The single 22-element prism
witness works for every $1\leq k\leq7$.
Both bounds follow from the displayed formulas and
Lemma~\ref{lem:small-generators}.
\end{remark}

Work in either of the stated languages.  In the expanded
language include the three absorbing-zero laws in the
ambient ai-semiring axioms.  Let $\mathcal{P}$ be the variety of ai-semirings defined by
\[
 \{\mathbf{w}_0\preceq\mathbf{u}_m:m\geq2\},
\]
and let
\[
 \mathcal{L}=\mathsf{V}\{S_F: m\geq2,\ \varnothing\ne F\subseteq V(\Gamma_m),
                         \ |V(\Gamma_m)\setminus F|\geq2\}.
\]

\begin{theorem}\label{thm:full-interval}
There are inclusions
\[
 \mathcal{L}\subseteq\mathsf{V}(B_2(\mathbb{Z}))\cap
 \bigcap_{p\text{ prime}}\mathsf{V}(B_2(C_p))
 \subseteq\mathcal{B}\subsetneq\mathcal{P}.
\]
Moreover $\mathcal{L}\subseteq\mathsf{V}(B_I(G))$ whenever
$|I|\geq2$ and $G\ne\{1\}$.
Every variety $\mathcal W$ with
\[
 \mathcal{L}\subseteq\mathcal W\subseteq\mathcal{P}
\]
has no identity basis with a uniformly bounded number of
variables. This applies to $\mathcal{L}$, $\mathcal{B}$,
$\mathcal{P}$, and the displayed intersection of cyclic
Brandt varieties.
The variety $\mathcal{L}$ is locally finite.
\end{theorem}

\begin{proof}
The local representations and the nontrivial cyclic-subgroup
argument prove all inclusions except strictness at the top.
Each small generated subsemiring of $W_k$ embeds into a
generator of $\mathcal{L}$, so belongs to $\mathcal W$.
On the other hand $W_k$ violates its own prism identity,
and hence lies outside $\mathcal{P}$ and outside $\mathcal W$.
This proves the basis assertion for the entire interval.

The generators of $\mathcal{L}$ satisfy
\begin{equation}\label{eq:nilpotent-laws}
 x^2\approx y^2,\qquad xyzt\approx x^2.
\end{equation}
In the named-zero language they satisfy $x^2\approx0$
and $xyzt\approx0$.
Every semiring term expands, by distributivity, into an
idempotent sum of nonempty multiplicative words.  The laws
\eqref{eq:nilpotent-laws} replace every word of length at
least four by the common square.  On finitely many
generators there are only finitely many words of lengths
at most three and only finitely many subsets to sum.
Thus $\mathcal{L}$ is locally finite.

For strictness, every flat semiring satisfies
\begin{equation}\label{eq:flat-law}
 u+yu+yv\approx v+yu+yv.
\end{equation}
If $u=v$ this is immediate; if $u\ne v$, then
$yu+yv=y(u+v)=0$, which absorbs the remaining sum.
Consider instead the three-element semiring
$D=\{a<b<\infty\}$, with addition maximum, and
\[
 xy=\begin{cases}a,&x,y\in\{a,b\},\\
                  \infty,&\text{otherwise}.
     \end{cases}
\]
Multiplication is associative, and each of its left and
right translations preserves maximum; hence $D$ is an
ai-semiring with absorbing element $\infty$.
It satisfies every prism inequality: if any variable is
$\infty$, it occurs in a retained word, so the right-hand
term is $\infty$ and the inequality holds; otherwise all
triple products are $a$.  But \eqref{eq:flat-law} fails at
$u=b,v=y=a$.  Thus $D\in\mathcal{P}\setminus\mathcal{B}$.
The same example works with $0$ named as $\infty$.
\end{proof}

\begin{corollary}[Sparse prism systems]\label{cor:sparse}
Let $M$ be an unbounded set of integers greater than one,
and let $\mathcal{P}_M$ be the variety of ai-semirings defined by
\[
 \{\mathbf{w}_0\preceq\mathbf{u}_m:m\in M\}.
\]
Every variety in $[\mathcal{L},\mathcal{P}_M]$ has no
bounded-variable identity basis.
\end{corollary}

\begin{proof}
For a proposed variable bound $k$, choose $m\in M$
with $k\leq4m-1$.  The same witness $S_{H_m}$ has
all its small generated subsemirings in $\mathcal{L}$, and
fails an identity of $\mathcal{P}_M$.
\end{proof}

The ai-semiring variety defined by the two laws
\eqref{eq:nilpotent-laws} is finitely based and locally
finite by the same word argument, and contains $\mathcal{L}$.

For an ai-semiring variety $\mathcal V$ and an integer
$N\geq1$, let $\nu_{\mathcal V}(N)$ be the least
integer $k\geq0$ such that every ai-semiring $A$ with
$|A|\leq N$ and $A\notin\mathcal V$ fails an identity
of $\mathcal V$ involving at most $k$ distinct variables.
The identities used may depend on $A$; the bound is uniform.
In the expanded signature the ambient algebras have the named
common absorbing zero.  If there are no such counterexamples,
set $\nu_{\mathcal V}(N)=0$.

\begin{theorem}\label{thm:linear-variable-complexity}
For every variety $\mathcal{L}\subseteq\mathcal V\subseteq\mathcal{P}$
and every $N\geq22$,
\[
 4\left\lfloor\frac{N-2}{10}\right\rfloor
 \ \leq\ \nu_{\mathcal V}(N)\ \leq\ N.
\]
Thus the variable complexity grows linearly in $N$,
uniformly throughout this interval, in both signatures.
\end{theorem}

\begin{proof}
For the upper bound, take any identity of $\mathcal V$
that fails in an algebra $A$ of size at most $N$.
A witnessing assignment uses at most $N$ distinct elements.
Identify variables assigned the same element, replacing them
by one new variable for each distinct assigned value.
The substituted identity is still valid in $\mathcal V$,
still fails in $A$, and has at most $N$ variables.

For the lower bound, put $m=\lfloor(N-2)/10\rfloor\geq2$.
The prism witness $S_{H_m}$ has $10m+2\leq N$ elements
and lies outside $\mathcal{P}$, hence outside $\mathcal V$.
The local-capture result used in
Theorem~\ref{thm:full-interval} puts every subsemiring of
$S_{H_m}$ generated by at most $4m-1$ elements in
$\mathcal{L}\subseteq\mathcal V$.
It follows that this witness satisfies every identity of
$\mathcal V$ involving at most $4m-1$ variables.
An identity excluding it must therefore use at least $4m$
variables.  This proves the stated lower bound.
\end{proof}

The length of an identity is the total number of variable,
constant and operation-symbol occurrences in its two terms.
For each finite $N$, there is a finite uniform length bound
for identities excluding all algebras of size at most $N$
outside a fixed variety: in a finite signature there are only
finitely many isomorphism types of such algebras.
The length of an identity is at least its number of distinct
variables. Hence, for every variety in
Theorem~\ref{thm:linear-variable-complexity}, this uniform
length bound is at least
\[
 4\left\lfloor\frac{N-2}{10}\right\rfloor
 \qquad(N\geq22),
\]
and therefore grows at least linearly in $N$.

For comparison, a finite flat group with a nonabelian Sylow
subgroup has term-length equational complexity bounded above
by a constant multiple of $(\log N)^3$ for sufficiently
large $N$, also in the binary semiring language
\cite[Theorem 14 and Remark 16]{Jackson19}.
The linear lower bound here applies after passing to a
Brandt semiring with at least two indices, even for the
same group.  Thus the matrix-index obstruction produces
a quantitative difference from the flat-group obstruction.

\subsection{Colored graph identities}
\label{sec:graphs}

The prism is one member of a broader graph construction.
Local representations follow from a cut condition, while
universal noncommutative identities are characterized by
the colored graph structure.

Let $\Gamma$ be a finite connected simple cubic bipartite
graph with a proper edge coloring $A,B,C$, and put
$N=|V(\Gamma)|$.  Each color class is a perfect matching
of size $N/2$.  Each two-color spanning subgraph is a
disjoint union of even cycles.  Let $c_{AB},c_{BC},c_{CA}$
be their numbers of cycles, and put
\[
 c(\Gamma)=c_{AB}+c_{BC}+c_{CA}.
\]
A vertex-star is always read in the order $abc$, with
$a\in A,b\in B,c\in C$.

Define a graph $Q$ whose vertices are the bicolored cycles,
retaining their three types.  Each variable is an edge of
$Q$, oriented as follows:
\[
 A:CA\longrightarrow AB,\qquad
 B:AB\longrightarrow BC,\qquad
 C:BC\longrightarrow CA.
\]
For example, an $A$-edge joins the unique $CA$-cycle
containing it to the unique $AB$-cycle containing it.
Every vertex of $\Gamma$ supplies a triangle in $Q$
whose directed boundary reads $abc$.
The triangles belonging to adjacent vertices of $\Gamma$
share their common variable edge.  Hence $Q$ is connected.
There are
\[
 |V(Q)|=c(\Gamma),\qquad |E(Q)|=3N/2,
\]
and $N$ designated triangles.  Parallel edges in $Q$
are allowed.  At a vertex of $Q$, the incident triangle
corners occur in the cyclic order supplied by that
bicolored cycle of $\Gamma$.

\begin{lemma}[Tree--dual-tree inequality]\label{lem:tree-dual}
For every spanning tree $T$ of $Q$, the graph on
$V(\Gamma)$ using exactly the edges outside $T$ is
connected.  Consequently
\begin{equation}\label{eq:colored-bound}
 c(\Gamma)\leq N/2+2.
\end{equation}
If equality holds, the complementary edges themselves
form a spanning tree $T^*$ of $\Gamma$.
The inequality also holds for connected properly
three-colored cubic bipartite multigraphs.
\end{lemma}

\begin{proof}
If the complementary graph were disconnected, take a
nonempty proper vertex set $Y$ equal to one component.
Its nonempty cut $\partial_\Gamma Y$ would consist
entirely of edges of $T$.  At a vertex of $Q$, the
number of incident cut edges is even: around the
corresponding bicolored cycle one enters and leaves
$Y$ equally often.  Thus the cut would be a nonempty
even-degree subgraph of a tree.  This is impossible,
since a nonempty finite forest has a vertex of degree one.
The complementary graph is connected and therefore has
at least $N-1$ edges.  Since $|T|=c(\Gamma)-1$,
\[
 3N/2-(c(\Gamma)-1)\geq N-1,
\]
which gives \eqref{eq:colored-bound}.  Equality makes the
connected complementary graph a tree.  None of this
argument excludes parallel edges in the original graph.
\end{proof}

\begin{theorem}[A colored Euler criterion]
\label{thm:colored-identity}
Suppose
\begin{equation}\label{eq:colored-equality}
 c_{AB}+c_{BC}+c_{CA}=N/2+2.
\end{equation}
Delete any one star $v_0$, put
$\mathbf{u}_\Gamma=\sum_{v\ne v_0}\mathbf{w}_v$, and let
$\mathbf{w}_{v_0}$ be the missing star word.  Then for every
group $G$ and every nonempty Brandt index set $I$,
\[
 B_I(G)\models \mathbf{w}_{v_0}\preceq\mathbf{u}_\Gamma.
\]
\end{theorem}

\begin{proof}
Only an evaluation with $\mathbf{u}_\Gamma\ne0$ needs attention.
Every retained word then has a common value
$(s,\delta,t)$.  Each variable still occurs at its
other graph endpoint, so every variable is nonzero.
The $AB$-cycle through $v_0$, with $v_0$ deleted,
gives a path of retained equations identifying the right
index of the missing $A$-letter with the left index of
its $B$-letter.  The $BC$-cycle gives the other
internal index equality.  The other endpoints of the
missing $A$- and $C$-edges give outside indices
$s,t$.  Thus the identity reduces to equality of the group labels.

Replace every $A$-label $a$ by $\delta^{-1}a$.
All retained triangle products are now $1$.
Choose a spanning tree $T$ of $Q$.  For an oriented
edge $u\to v$, change its label by
\[
 g_{uv}\longmapsto h(u)^{-1}g_{uv}h(v).
\]
Each directed triangle product changes by conjugation
at its starting vertex, so all retained equations with
value $1$ survive.  Choose the $h(u)$ recursively
along $T$ to make every label on $T$ equal to $1$.
This requires only multiplication and inversion in $G$.

By Lemma~\ref{lem:tree-dual}, the complementary edges
form a dual tree $T^*$.  Root it at the missing star.
Process its nonroot leaves toward the root.  At the
current leaf triangle, every edge except its parent
edge is either in $T$, or is a previously processed
child edge, and therefore has label $1$.  The retained
triangle equation forces its parent-edge label to be
$1$.  Induction treats every edge of $T^*$.
Every edge label is now $1$, so the missing triangle
product is $1$ too.  Undoing the label changes gives
its original product $\delta$.  This proves the
identity without a commutativity assumption.
\end{proof}

The colored equality is also necessary for this identity
over any nonabelian group.  The following gives an exact
characterization, without using a classification of surfaces.

\begin{theorem}[Exact graph--group criterion]
\label{thm:graph-group-exact}
For the graph $\Gamma$ above, any choice of the missing
star $v_0$, any group $G$, and any nonempty index set
$I$,
\[
 B_I(G)\models \mathbf{w}_{v_0}\preceq\mathbf{u}_\Gamma
 \quad\Longleftrightarrow\quad
 G\text{ is abelian}\ \text{or}\ c(\Gamma)=N/2+2.
\]
If $c(\Gamma)<N/2+2$, a counterevaluation exists in
every nonabelian group corner.
\end{theorem}

\begin{proof}
For every graph under consideration, the index propagation
and normalization of $A$-labels from the preceding proof
reduce validity to the following implication in $G$:
\begin{equation}\label{eq:graph-group-implication}
 a_vb_vc_v=1\quad(v\ne v_0)
 \quad\Longrightarrow\quad a_{v_0}b_{v_0}c_{v_0}=1.
\end{equation}
Conversely, failure of this implication yields a Brandt
counterevaluation by putting all labels in a single
diagonal corner.  Every retained word is then $(i,1,i)$;
the missing word is a different nonzero element, and its
addition changes the value to zero.

If $G$ is abelian, give the two bipartition classes
signs $\varepsilon(v)=1,-1$.  In
\[
 \prod_{v\in V(\Gamma)}
       (a_vb_vc_v)^{\varepsilon(v)}
\]
every edge label occurs once positively and once negatively.
The product is $1$, so if all retained factors are $1$,
the missing factor is $1$.  This proves the abelian case.
The case of equality in \eqref{eq:colored-bound} is
Theorem~\ref{thm:colored-identity}.

It remains to disprove \eqref{eq:graph-group-implication}
in a nonabelian group when the inequality is strict.
Choose a spanning tree $T$ of the corner graph $Q$.
By Lemma~\ref{lem:tree-dual}, choose a dual spanning tree
$T^*$ of $\Gamma$ using only edges outside $T$.
Identify these dual edges with their variable edges.
There are
\begin{equation}\label{eq:graph-deficit}
 r=\frac{3N}{2}-(c(\Gamma)-1)-(N-1)
   =\frac N2+2-c(\Gamma)>0
\end{equation}
edges outside $T\cup T^*$.

\medskip
\noindent\emph{The polygon obtained from the dual tree.}
Take one abstract triangle for each vertex of $\Gamma$.
Its three corner types are $AB,BC,CA$; its sides are
labelled by its three incident variables.  Pair equal
variable sides, matching their endpoint types.  The
resulting corner classes are precisely the vertices of
$Q$.  Around such a class, the successive triangle
corners follow its bicolored cycle; in particular its
corner link is a single cycle.
Orient the triangles in one bipartition class with
boundary $abc$, and those in the other with boundary
$c^{-1}b^{-1}a^{-1}$.  A paired side is therefore read
in opposite directions by its two triangles.

Initially perform only the side gluings specified by
$T^*$.  They produce a disk: root $T^*$, and attach
each fresh triangle along its unique parent side, making
no further side or vertex identifications at that step.
Its boundary is an edge-paired polygon in which each
edge outside $T^*$ occurs twice, once in each direction.
Pairing the remaining sides recovers the preceding
triangle system.

We need a precise observation about contraction in an
oriented edge-paired polygon.  If a paired side represents
an edge with distinct endpoint vertex classes in the
quotient, deleting its two boundary occurrences merges
exactly these two classes and leaves every other class
unchanged.  Consider the cyclic corner links at the two
endpoints.  Removing the edge
cuts these two distinct cycles into paths.  The new
boundary adjacencies splice the two paths into one cycle,
while all other corner links remain unchanged.  This is
exactly the corner identification induced by deleting
the two side occurrences.

Apply this observation successively to the edges of $T$.
Each such edge still has distinct endpoints when it is
contracted: otherwise its endpoints would already be
joined by contracted edges of the tree, creating a cycle
in $T$.  After all $c(\Gamma)-1$ contractions, the
remaining cyclic boundary word $\mathbf{w}$ has length $2r>0$.
Each of its $r$ letters occurs once positively and once
negatively, and its own side pairing has exactly one
vertex class.  In particular, no extra vertex
identifications are being retained after the $T$-letters
are deleted.

\medskip
\noindent\emph{An interleaving pair of letters.}
The cyclic word $\mathbf{w}$ has no adjacent inverse pair.
For adjacent occurrences $xx^{-1}$, or $x^{-1}x$,
the common polygon corner is identified only with itself:
its two incident sides are the paired occurrences, and
their identification fixes that corner.  It would form
a vertex class separate from the other corners,
contradicting the one-vertex property.

The matching of the two occurrences of each letter around
the boundary must therefore contain two interleaving
pairs.  Indeed, if it were noncrossing, choose a matched
pair with the shortest linear interval between its
endpoints.  If that interval contained an endpoint, its
mate would also lie inside the interval, since otherwise
the pairs would cross.  That would give a shorter matched
interval.  Thus a noncrossing matching has an adjacent
matched pair, which here would be an adjacent inverse pair.

Choose interleaving letters $x,y$.  Assign all other
letters of $\mathbf{w}$ the group identity.  The remaining
word is, up to cyclic rotation,
\[
 x^\epsilon y^\eta x^{-\epsilon}y^{-\eta},
 \qquad \epsilon,\eta\in\{1,-1\}.
\]
Choose noncommuting elements $g,h\in G$ and assign
$x=g,\ y=h$.  Their indicated powers do not commute
either, so this value is nonidentity.  Cyclic rotation
only conjugates a product and cannot change that fact.

\medskip
\noindent\emph{Solving all retained relations.}
Set every $T$-edge label to $1$, and give the edges
outside $T\cup T^*$ the labels just chosen.
Root $T^*$ at the missing triangle.  Process the other
triangles from the leaves towards the root.  At each
step the parent edge is the only undetermined side of
the current triangle; its retained boundary equation
therefore determines that label uniquely.

After these substitutions, the missing boundary product
is the boundary product of the glued disk, up to a
cyclic change of starting point and possible inversion.
This follows one gluing at a time: with compatible
starting points, let the parent boundary read $\mathbf{u}t$
and the child boundary read $t^{-1}\mathbf{v}$.
The child equation gives $t=\mathbf{v}$, so substitution changes
the parent boundary to $\mathbf{u}\mathbf{v}$, exactly the boundary
after gluing.  Repeating this substitution eliminates
all nonroot triangles.
Deleting the $T$-letters now leaves the word $\mathbf{w}$,
because all their labels were $1$.
Thus every retained relation holds, while the missing
boundary has the nonidentity value just constructed,
up to conjugacy or inversion.  This disproves
\eqref{eq:graph-group-implication} in $G$ and completes
the equivalence.
\end{proof}

\begin{corollary}[The cut condition over abelian groups]
\label{cor:abelian-graph}
Suppose $\Gamma$ satisfies the cut condition
\eqref{eq:cut-condition}, without assuming
\eqref{eq:colored-equality}, and let $H$ be obtained by
deleting one star.  Then $|S_H|=5N/2+2$.
Every subsemiring generated by at most $N-2$ elements
belongs to $\mathsf{V}(B_2(C_r))$ for every $r\geq2$, and to
$\mathsf{V}(B_2(\mathbb{Z}))$.
For every abelian group $G$ and nonempty $I$,
\[
 S_H\notin\mathsf{V}(B_I(G)).
\]
\end{corollary}

\begin{proof}
The size follows from \eqref{eq:SH-carrier}.
Lemma~\ref{lem:capture-criterion} puts a nonempty
$k$-generated subsemiring into some $S_F$ with
$1\leq|F|\leq k\leq N-2$.
At least two stars are therefore omitted, so the local
representation theorem applies.  In the named-zero
language the empty-generated algebra is $\{0\}$.
Finally, the natural assignment in $S_H$ has
$\mathbf{u}_\Gamma=h$ and $\mathbf{w}_{v_0}=0$, whereas
Theorem~\ref{thm:graph-group-exact} gives the identity
in every Brandt semiring over an abelian group.
\end{proof}

\begin{lemma}\label{lem:distinct-cycles}
Suppose $\Gamma$ satisfies \eqref{eq:colored-equality}
and has no edge cut of size less than three.  The two
endpoints of a $B$-edge belong to distinct $AC$-cycles.
\end{lemma}

\begin{proof}
Suppose a $B$-edge $uv$ joined two vertices of one
$AC$-cycle.  Delete $u,v$, reconnect their two
$A$-neighbors by an $A$-edge, and reconnect their
two $C$-neighbors by a $C$-edge.  Proper coloring
and simplicity ensure that this creates no loops; parallel
edges may occur.  The result is properly colored, cubic,
and bipartite.
It is connected: if $\Gamma-\{u,v\}$ had at least two
components, each would have at least three boundary
edges in $\Gamma$, although only four edges leave
$\{u,v\}$.  This is impossible.

The $AB$- and $BC$-cycles containing $uv$ are
shortened but their numbers do not change.
The vertices $u,v$ have opposite bipartition signs,
so both paths between them on their $AC$-cycle have
odd length.  One has an $A$-edge at each end and the
other a $C$-edge at each end.  The reconnections close
these paths separately, splitting one $AC$-cycle into
two.  Thus the new graph has $N'=N-2$ vertices and
$c'=c(\Gamma)+1$.  It satisfies
\[
 c'-N'/2=(c(\Gamma)+1)-(N-2)/2=4,
\]
contrary to the multigraph version of
Lemma~\ref{lem:tree-dual}.  The supposed edge cannot exist.
\end{proof}

\begin{theorem}[Uniform graph witnesses]\label{thm:graph-witness}
Suppose $\Gamma$ satisfies the cut condition
\eqref{eq:cut-condition} and the colored equality
\eqref{eq:colored-equality}.  Delete one star to obtain
$H$.  Then
\[
 |S_H|=5N/2+2,\qquad S_H\notin\mathcal{B}.
\]
Every subsemiring of $S_H$ generated by at most $N-1$
elements belongs to $\mathsf{V}(B_2(C_r))$ for every
$r\geq2$, and to $\mathsf{V}(B_2(\mathbb{Z}))$.
\end{theorem}

\begin{proof}
Each color matching has $N/2$ edges, all still appearing
after one star is deleted.  Formula
\eqref{eq:SH-carrier} gives the size.
The natural assignment makes $\mathbf{u}_\Gamma=h$ and the
missing word zero in $S_H$, whereas
Theorem~\ref{thm:colored-identity} gives the identity in
all Brandt semirings.  Hence $S_H\notin\mathcal{B}$.

For the local claim, use the notation $X,Z,L,M$ of
Section~\ref{sec:capture}, with $1\leq|X|\leq N-1$.
The empty-generated algebra in the named-zero language
is $\{0\}$ and already belongs to the required varieties.
As in Lemma~\ref{lem:small-generators}, it suffices to
find a retained vertex $w\notin M$ such that its
possible edge to $v_0$ is not in $Z$.
Suppose no such vertex exists.
Every $B$-edge other than the one at $v_0$ has a
retained endpoint in the same bipartition class as
$v_0$, hence not adjacent to $v_0$.  This endpoint
must be mandatory, forcing the $B$-edge to be an
actual letter generator.  The other endpoint $w_B$
of the $B$-edge at $v_0$ forces that edge too,
either through being mandatory or through its index
being in $Z$.  Thus all $N/2$ $B$-indices lie
in $Z$.

Every vertex outside $\{v_0,w_B\}$ is covered by
an $A$- or $C$-index in $Z$, either because it
is mandatory or because of its selected edge to $v_0$.
By Lemma~\ref{lem:distinct-cycles}, the two exceptions
lie on different even $AC$-cycles.  On either of
these cycles, covering all but one vertex requires at
least half the cycle length in selected edges.  On
every other cycle all vertices are covered, again
requiring at least half the length.  Hence at least
$N/2$ $A/C$-indices lie in $Z$ as well.
This gives $|X|\geq|Z|\geq N$, a contradiction.

A permitted second deletion therefore exists.  The
capture criterion embeds $\langle X\rangle$ into
$S_F$ with exactly two stars omitted.  The local
representation theorem, already proved for every graph
satisfying the cut condition, completes the proof.
\end{proof}

For the prism coloring,
$c_{AB}=m$, $c_{BC}=m$, and $c_{CA}=2$.
Indeed the first two two-color subgraphs each consist
of $m$ four-cycles, and the last consists of the two
horizontal $2m$-cycles.  Thus
\eqref{eq:colored-equality} holds with $N=4m$,
recovering the earlier construction and its sharp
second-deletion threshold.

The cut condition alone does not give the noncommutative
identity.  Let $\Gamma=K_{3,3}$, with positive vertices
$u_i$ and negative vertices $v_j$, indices modulo
three, and edges
\[
 a_i:u_i v_i,\qquad b_i:u_i v_{i+1},\qquad
 c_i:u_i v_{i+2}.
\]
Its positive stars are $a_i b_i c_i$, and its negative
stars are $a_j b_{j-1}c_{j-2}$.  Delete $u_0$.
For noncommuting group elements $x,z$, assign
\[
\begin{array}{c|ccc}
 i&a_i&b_i&c_i\\ \hline
 0&1&x&z^{-1}x^{-1}z\\
 1&z^{-1}&z&1\\
 2&z^{-1}x&1&x^{-1}z
\end{array}
\]
The two retained positive products are
$z^{-1}z=1$ and $z^{-1}x\,x^{-1}z=1$.
The three negative products are respectively
\[
 1,\qquad z^{-1}x\,x^{-1}z=1,\qquad
 z^{-1}x z\,z^{-1}x^{-1}z=1.
\]
The missing product is $xz^{-1}x^{-1}z\ne1$.
Taking $x=(12)$ and $z=(23)$ gives a
counterevaluation in a corner isomorphic to $S_3$.

Nevertheless $K_{3,3}$ satisfies the cut condition.
For a vertex set containing $p$ positive and $q$
negative vertices, its cut has size
$3(p+q)-2pq$.  On taking the smaller shore, assume
$1\leq p+q\leq3$.  The value is three only for a
singleton; for totals two and three it is at least four
and five, respectively.  Each bicolored subgraph is one
six-cycle, so $c(\Gamma)=3$, whereas $N/2+2=5$.

\section{Flat group extensions}
\label{sec:unbounded-flat-groups}

\subsection{Commutative witnesses}

A commutative version of the prism construction gives a
second uniform obstruction.  Its local representations use
an infinite cyclic group instead of matrix indices.
The ordinary nonfinite basability of $\flat(\mathbb{Z})$ is
known; see \cite{MS25,GRZ26}.  The construction below gives
the stronger bounded-variable conclusion, simultaneous
finite witnesses, and the conclusion for every group
without a finite exponent.

Let $H$ be a finite linear three-uniform hypergraph
without isolated vertices.  In this subsection its edges
are unordered triples of distinct vertices.
On the disjoint union
\[
 C_H=V(H)\dot\cup\{p_x:x\in V(H)\}\dot\cup\{h,0\}
\]
put flat addition and commutative multiplication, with
the following nonzero products:
\begin{equation}\label{eq:commutative-hypergraph-products}
 xy=p_z\quad(\{x,y,z\}\in H),\qquad xp_x=h.
\end{equation}
All products not specified by these rules or their
commutative counterparts are zero.

\begin{lemma}\label{lem:commutative-hypergraph}
The algebra $C_H$ is a commutative flat ai-semiring.
Every square and every product of four elements is zero.
A product of three elements is nonzero precisely when
its factors are the three vertices of an edge; its value
is then $h$.
\end{lemma}

\begin{proof}
Linearity makes the multiplication unambiguous.
A nonzero product $(xy)z$ must have $xy=p_z$;
therefore its three factors are vertices forming an
edge, and its value is $h$.
Conversely an edge gives this value in either bracketing.
Every other triple product is zero in both bracketings.
This proves associativity and the asserted triple rule.
The disjoint types in the construction give the square
and four-factor assertions.

For zero-cancellation, fix a vertex $x$.
Equal nonzero products $xy=xy'=p_z$ would give two
edges containing $x,z$, so linearity gives $y=y'$.
The output $h$ occurs only with the factor $p_x$.
A factor $p_x$ has just one nonzero partner, namely $x$.
These observations prove zero-cancellation; commutativity
gives it on both sides.
As in Lemma~\ref{lem:SH}, zero-cancellation is exactly
what is needed for flat addition to distribute.
\end{proof}

Apply this construction to the prism hypergraph $H_m$
of \eqref{eq:prism-faces}, with $U_0$ deleted, forgetting
only the order inside each edge.  It has $6m$ vertices,
so
\begin{equation}\label{eq:commutative-prism-size}
 |C_{H_m}|=12m+2.
\end{equation}
Keep the words in $\mathbf{u}_m$ ordered as in the original
prism identity.  Lemma~\ref{lem:prism-identity} shows that
this identity holds in every flat group, and indeed in
every Brandt semiring.  In $C_{H_m}$, the identically
named evaluation gives $\mathbf{u}_m=h$ and $a_0b_0c_0=0$.
Thus
\begin{equation}\label{eq:commutative-prism-failure}
 C_{H_m}\not\models a_0b_0c_0\preceq\mathbf{u}_m.
\end{equation}

\begin{lemma}\label{lem:commutative-local-representation}
Let $\Gamma$ satisfy \eqref{eq:cut-condition}, and let
$F$ be a nonempty set of its vertex-stars omitting
at least two graph vertices.  Then
\[
 C_F\in\mathsf{V}(\flat(\mathbb{Z})).
\]
All constructions preserve the absorbing zero.
\end{lemma}

\begin{proof}
Use the integral relation module from
\eqref{eq:relation-module} and put
\[
 K=\mathbb{Z}^{E_F}/W_F.
\]
Lemma~\ref{lem:primitive} makes $K$ a finitely generated
free abelian group.  Write $[x]$ for the class of $e_x$.
The short-support lemma shows that distinct vertices have
distinct classes, and, for any three vertices with
repetition allowed,
\begin{equation}\label{eq:commutative-triple-test}
 [x]+[y]+[z]=0
 \quad\Longleftrightarrow\quad
 \{x,y,z\}\text{ is an edge of }F.
\end{equation}
Indeed, a repeated vertex gives a nonzero integral vector
with support at most two, which cannot belong to $W_F$.
For three distinct vertices, a support-three relation must
be a scalar multiple of a vertex-star.  Its three
coefficients are $1$, so that scalar is $1$.

In the flat extension of the additive group $\mathbb{Z}\times K$,
generate a subsemiring $R$ by zero and
\[
 A_x=(1,[x])\qquad(x\in E_F).
\]
Every nonzero element of $R$ is a nonempty product of
these generators, and its first coordinate is its length.
Let $H_0=(3,0_K)$.
For an incident edge $\{x,y,z\}$,
\[
 A_xA_yA_z=H_0,\qquad
 A_yA_z=(2,-[x])=:P_x.
\]
Thus $H_0$, every $A_x$, and every $P_x$ belong to
$R$ and divide $H_0$.

Adjoin a formal identity for the divisibility criterion, and
let $J$ consist of zero and all elements of $R$
which do not divide $H_0$.
If a product involving an element divides $H_0$,
then that element also divides $H_0$.
Hence $J$ is a multiplicative ideal.
Flat addition also makes $J+R\subseteq J$, so its
Rees quotient is a semiring.

The first coordinate identifies all surviving elements.
At degree one they are the $A_x$.
A degree-two element $A_yA_z$ survives precisely when
some $A_x$ completes it to $H_0$; by
\eqref{eq:commutative-triple-test}, it is then $P_x$.
At degree three only $H_0$ survives, and every element
of degree at least four lies in $J$.
Distinctness within the first two types follows from
$[x]\ne[y]$; their degrees separate the three types.
Finally
\[
 A_xP_y=H_0\quad\Longleftrightarrow\quad [x]=[y]
 \quad\Longleftrightarrow\quad x=y.
\]
These are exactly the multiplication rules
\eqref{eq:commutative-hypergraph-products}.
The same-symbol correspondence therefore gives
$R/J\cong C_F$.

Since $\mathbb{Z}\times K$ is free abelian of finite rank,
the zero-smash construction gives
$\flat(\mathbb{Z}\times K)\in\mathsf{V}(\flat(\mathbb{Z}))$.
The subsemiring and quotient above prove the assertion.
\end{proof}

Write $N_+=(\mathbb Z_{>0},+)$ for the positive cyclic
semigroup.  Its operation becomes multiplication in
$\flat(N_+)$; the adjoined absorbing zero is a new
symbol, distinct from every integer.

\begin{corollary}\label{cor:commutative-scalar-representation}
Each $C_F$ in Lemma~\ref{lem:commutative-local-representation}
is a Rees quotient of a subsemiring of $\flat(N_+)$,
and hence also of a subsemiring of $\flat(\mathbb{Z})$.
\end{corollary}

\begin{proof}
Work in the free abelian group $K$ from the preceding proof.
The forbidden vectors $[x]-[y]$, for $x\ne y$,
and $[x]+[y]+[z]$, for every nonedge triple with
repetitions allowed, are nonzero by the short-support lemma.
There are only finitely many such vectors.
Choose a basis $b_1,\ldots,b_r$ of $K$, and define
$\ell_t(b_i)=t^{i-1}$ for an integer $t$.
For a nonzero vector $v=\sum_i c_i b_i$, the value
$\ell_t(v)=\sum_i c_it^{i-1}$ is a nonzero polynomial
in $t$, so it vanishes at only finitely many integers.
Choose $t$ avoiding the roots for all forbidden vectors,
and put $\lambda_x=\ell_t([x])$.
These integers are pairwise distinct and satisfy
\[
 \lambda_x+\lambda_y+\lambda_z=0
 \quad\Longleftrightarrow\quad
 \{x,y,z\}\text{ is an edge of }F.
\]
Put $L=\max_x|\lambda_x|$, and choose an integer $N>6L$.
In additive notation for $N_+$, generate its flat
subsemiring by the adjoined zero and $N+\lambda_x$, and
retain exactly the divisors of $3N$.
Every generator is positive, and any four generators
have sum at least $4(N-L)>3N$.
The same divisibility argument as above leaves precisely
the values
\[
 N+\lambda_x,\qquad 2N-\lambda_x,\qquad 3N.
\]
The choice of $N$ separates these types.
A pair survives precisely when a third generator makes
an edge, and a generator times a complementary pair
survives precisely for matching indices.
The resulting Rees quotient is $C_F$.
\end{proof}

\begin{lemma}\label{lem:commutative-capture}
For every nonempty $X\subseteq C_{H_m}$, there is a
nonempty set $F\subseteq H_m$ of at most $|X|$ stars
such that the same-symbol map embeds
$\langle X\rangle$ into $C_F$.
\end{lemma}

\begin{proof}
Let $L=X\cap V(H_m)$ be the actual vertex generators.
Let $Z$ consist of their indices and of the indices
$x$ for which $p_x\in X$; $h,0$ contribute no index.
Thus $L\subseteq Z$ and $|Z|\leq|X|$.
Call a retained star mandatory if it contains at least
two members of $L$, and let $M$ be their set.
Writing $L_M=L\cap E_M$, the degree bound two gives
\[
 2|M|\leq\sum_{x\in L_M}\deg_M(x)\leq2|L_M|.
\]
For each $z\in Z\setminus E_M$, choose an incident
retained star, which exists because $H_m$ has no
isolated vertices.  Adjoin these stars to $M$ to
obtain $F$.  Then
\[
 |F|\leq |M|+|Z\setminus E_M|
 \leq |L_M|+|Z|-|Z\cap E_M|\leq |Z|\leq|X|.
\]
If this set is empty, choose one retained star instead.

No operation creates a new vertex.
Every newly generated $p_x$ comes from a product of
two actual vertex generators and hence from a mandatory
star.  Thus $F$ supplies every symbol in $\langle X\rangle$.
It preserves all nonzero vertex-pair products, and the
products $xp_x=h$ depend only on matching symbols.
Since $F\subseteq H_m$, no zero product becomes nonzero.
The same-symbol map therefore preserves multiplication.
It is injective and preserves zero, hence preserves flat
addition as well.
\end{proof}

\begin{theorem}[Uniform commutative witnesses]
\label{thm:commutative-uniform}
For $k\geq1$, put
\[
 m=\max\{2,\lceil(k+2)/4\rceil\},\qquad
 U_k=C_{H_{m}}.
\]
Then $U_k$ is a commutative flat semiring of exactly
$12m+2$ elements, lies outside $\mathcal{P}$, and every
subsemiring generated by at most $k$ elements belongs
to $\mathsf{V}(\flat(\mathbb{Z}))$.
In fact each such subsemiring is a semiring divisor of
$\flat(N_+)$, and hence of $\flat(\mathbb{Z})$.
These statements hold in both signatures.
\end{theorem}

\begin{proof}
The size and the failed prism identity were established
in \eqref{eq:commutative-prism-size} and
\eqref{eq:commutative-prism-failure}.
By the capture lemma, a nonempty at-most-$k$-generated
subsemiring embeds in some $C_F$ with $|F|\leq k$.
Since $4m\geq k+2$, at least two prism vertices
lie outside $F$.
Apply Lemma~\ref{lem:commutative-local-representation}.
The stronger divisor assertion follows from
Corollary~\ref{cor:commutative-scalar-representation},
by taking the inverse image of the embedded subsemiring
under its quotient map.
In the named-zero signature the empty-generated algebra
is $\{0\}$, which also has the required property.
\end{proof}

\subsection{Finite bases and variable complexity}

For the rank-one finite case we use the established theorem
of Jackson \cite{Jackson08}, recorded in
\cite[Lemma 2.3]{RJZL}:
\begin{equation}\label{eq:group-criterion}
 G\text{ finite}:\quad
 \flat(G)\text{ is finitely based}
 \quad\Longleftrightarrow\quad
 \text{all Sylow subgroups of }G\text{ are abelian}.
\end{equation}
This result concerns the binary semiring language;
Lemma~\ref{lem:constant} transfers it to the named absorbing
zero without a change of convention.

\begin{theorem}[The exact flat-group reduction]
\label{thm:flat-group-all}
Let $G$ be any group.  If $G$ has no finite exponent,
the identities of $\flat(G)$ have no basis with a
uniform finite bound on the number of variables.
For arbitrary $G$,
\[
 \flat(G)\text{ is finitely based}
 \quad\Longleftrightarrow\quad
 \begin{gathered}
 G\text{ has finite exponent, and its group}\\
 \text{quasi-equational theory is finitely axiomatizable}.
 \end{gathered}
\]
The equivalence concerns both semiring signatures in
this paper.
\end{theorem}

\begin{proof}
Suppose $G$ has no finite exponent.
For each $d\geq1$, choose $g_d\in G$ with $g_d^d\ne1$.
The tuple $g=(g_d)_{d\geq1}\in G^{\mathbb N}$ has
infinite order: its $d$-th coordinate rules out $g^d=1$.
The cyclic subgroup generated by $g$, together with zero,
and the zero-smash lemma give
\[
 \flat(\mathbb{Z})\in
 \mathsf{V}(\flat(G^{\mathbb N}))\subseteq\mathsf{V}(\flat(G)).
\]
Consequently $U_k$ satisfies every at-most-$k$-variable
identity of $\flat(G)$.
But its prism identity holds in $\flat(G)$ and
fails in $U_k$.  No basis with variable bound $k$
can therefore exist.
For finite exponent, the asserted equivalence is the
established correspondence \cite[Lemma 2.2]{RJZL}.
Lemma~\ref{lem:constant} transfers finite basability to
the named absorbing-zero signature; the witnesses
already preserve that constant.
\end{proof}

\begin{corollary}\label{cor:commutative-variable-interval}
Every variety
\[
 \mathsf{V}(\flat(N_+))\subseteq\mathcal W\subseteq\mathcal{P}
\]
has no bounded-variable identity basis.  For $N\geq26$,
\[
 4\left\lfloor\frac{N-2}{12}\right\rfloor-1
 \ \leq\ \nu_{\mathcal W}(N)\ \leq\ N.
\]
Thus the variable complexity grows linearly in $N$, and
the uniform term-length bound grows at least linearly.
This applies to the flat extension of every group of
unbounded exponent, in both signatures.
\end{corollary}

\begin{proof}
The uniform commutative witness in
Theorem~\ref{thm:commutative-uniform} lies outside
$\mathcal{P}$, whereas each of its at-most-$k$-generated
subsemirings belongs to $\mathsf{V}(\flat(N_+))$.
This excludes any basis with variable bound $k$.
For the quantitative statement put
$m=\lfloor(N-2)/12\rfloor\geq2$.
The witness $C_{H_m}$ has $12m+2\leq N$ elements
and satisfies every identity of $\mathcal W$ in at
most $4m-2$ variables.  An excluding identity therefore
needs at least $4m-1$ variables.
The upper bound follows from the variable-identification
argument in the proof of
Theorem~\ref{thm:linear-variable-complexity};
the length comparison follows from the definitions.
If $G$ has unbounded exponent, the proof of
Theorem~\ref{thm:flat-group-all} gives
$\flat(\mathbb{Z})\in\mathsf{V}(\flat(G))$, so the lower
inclusion follows from $\flat(N_+)\leq\flat(\mathbb{Z})$.
The upper inclusion follows from the prism identity.
\end{proof}

\section{Naturally semilattice-ordered inverse semigroups}
\label{sec:natural-classification}

\subsection{Structural representations}
\label{sec:regular}

\begin{proposition}[The structure of regular flat semirings]
\label{prop:regular-structure}
Every nonzero regular flat semiring is an orthogonal sum
of Brandt semirings.  Equivalently, its multiplicative
semigroup is a primitive inverse semigroup.
\end{proposition}

\begin{proof}
Flat distributivity implies zero-cancellation.  Indeed,
if $ab=ac\ne0$ and $b\ne c$, then
$a(b+c)=0$ but $ab+ac=ab\ne0$; the right-sided
argument is the same.
For nonzero idempotents $e,f$ with $ef\ne0$,
\[
 e(ef)=ef=e\cdot f,\qquad (ef)f=ef=e\cdot f.
\]
Cancellation gives $ef=f$ and $ef=e$, respectively.
Thus distinct nonzero idempotents are orthogonal.

Regularity supplies an inverse: replace an inner inverse
$x$ of $a$ by $xax$.  If $b,c$ are both
inverses of nonzero $a$, the nonzero idempotents
$ab,ac$ satisfy $(ab)(ac)=ac$, so $ab=ac$.
Similarly $(ba)(ca)=ba$ gives $ba=ca$.  Hence
\[
 b=bab=(ca)b=c(ab)=c(ac)=cac=c.
\]
Zero has only zero as an inverse.  The semigroup is inverse,
and orthogonality makes every nonzero idempotent primitive.

We give the Brandt coordinates explicitly.
For nonzero $s,t$,
\[
 st=s\,d(s)r(t)\,t,
\]
so it is zero when $d(s)\ne r(t)$.  If these idempotents
agree, then
\[
 (st)(t^{-1}s^{-1})=s\,r(t)\,s^{-1}=ss^{-1}\ne0,
\]
so $st\ne0$.  The nonzero elements therefore form a
groupoid, with exact composability condition
$d(s)=r(t)$.

Partition the nonzero idempotents by $e\sim f$ if some
$s$ has $r(s)=e,d(s)=f$.  This is an equivalence
relation by idempotents, inverses, and composition.
In one component fix $e_0$.  For each $e_i$, choose
$u_i$ with $d(u_i)=e_0,r(u_i)=e_i$, taking
$u_0=e_0$, and let
\[
 G=\{g:d(g)=r(g)=e_0\}.
\]
This is a group.  The map
\[
 (i,g,j)\longmapsto u_i g u_j^{-1}
\]
is a bijection onto the component's nonzero elements: for
$r(s)=e_i,d(s)=e_j$ the inverse label is
$g=u_i^{-1}s u_j$, and $s=u_i g u_j^{-1}$.
It preserves Brandt multiplication, using
$u_j^{-1}u_j=e_0$ and orthogonality when middle indices
differ.  Distinct components multiply to zero.  The
zero-preserving bijection also preserves flat addition.
This proves the claimed decomposition.
\end{proof}

In this decomposition the Clifford case is precisely the
case where all index sets have size one.  A block with two
indices contains an off-diagonal element, which cannot
belong to a subgroup, whereas a rank-one block is a group
with the shared zero.

\begin{theorem}[All non-Clifford regular flat semirings]
\label{thm:regular-infinite}
Let $S$ be a regular flat semiring whose multiplicative
semigroup is not Clifford.  No cardinality restriction is
imposed.  Then, in either stated language,
\[
 S\text{ is finitely based}\quad\Longleftrightarrow\quad
 S\text{ is combinatorial}.
\]
In every negative case its identities have no basis with
a uniform variable bound.
\end{theorem}

\begin{proof}
By Proposition~\ref{prop:regular-structure}, write
$S=\bigoplus^{0} B_{I_\lambda}(G_\lambda)$.
Some index set has size at least two, so
Theorem~\ref{thm:orthogonal} identifies its variety with
$\mathsf{V}(B_2(G))$, where
$G=\prod_\lambda G_\lambda$.
This product group is trivial exactly when every
$G_\lambda$ is trivial.  Now apply
Corollary~\ref{cor:brandt-classification} and
Theorem~\ref{thm:all-groups}.
\end{proof}

\begin{theorem}[Complete finite regular flat classification]
\label{thm:finite-regular}
A finite regular flat semiring $S$ is finitely based if
and only if at least one of the following holds:
\begin{enumerate}[$(\rm i)$]
\item its multiplicative semigroup is combinatorial;
\item its multiplicative semigroup is Clifford and every
Sylow subgroup of every maximal subgroup is abelian.
\end{enumerate}
The criterion holds in both stated signatures.
\end{theorem}

\begin{proof}
Treat the zero-only algebra separately; it satisfies
$x\approx y$ and is finitely based.
Use the decomposition of Proposition~\ref{prop:regular-structure}.
If a block has at least two indices,
Theorem~\ref{thm:regular-infinite} proves the criterion.
If all blocks have one index, the variety is that of
$\flat(G)$, where $G$ is the
finite product of the block groups.  Its Sylow subgroups
are products of the corresponding Sylow subgroups of the
factors, so all are abelian exactly when the condition
holds in every factor.  Apply \eqref{eq:group-criterion}.
For a nonzero combinatorial rank-one sum, both operations
are the same semilattice operation.  It generates the
variety of the two-element semilattice with the operation
duplicated, and is finitely based as well.
\end{proof}

\begin{corollary}[The arbitrary-cardinality Clifford reduction]
\label{cor:clifford-reduction}
For an arbitrary Clifford regular flat semiring
$S=\bigoplus^{0} \flat(G_\lambda)$, its finite-basis
problem is exactly that of
$\flat(\prod_\lambda G_\lambda)$.
Writing $G=\prod_\lambda G_\lambda$, this
semiring is finitely based exactly when $G$
has finite exponent and its group quasi-equational theory
is finitely axiomatizable.
If $G$ has no finite exponent, there is no
bounded-variable identity basis.
\end{corollary}

\begin{proof}
The first assertion is Theorem~\ref{thm:orthogonal}.
For finite exponent use \cite[Lemma 2.2]{RJZL}.
The unbounded-exponent assertion follows from
Theorem~\ref{thm:flat-group-all}.
Lemma~\ref{lem:constant} treats the absorbing constant.
\end{proof}

The finite-exponent case follows from the group
quasi-equational correspondence, and the commutative
construction in Section~\ref{sec:unbounded-flat-groups}
settles the unbounded-exponent case.

The Clifford reduction admits a stronger form.  Here addition
is the meet in the natural order of an inverse semigroup,
and it need not be flat.
This representation and the resulting finite classification
belong to the framework of \cite{Jackson08,Jackson19}.
We give explicit projections and a proof in the binary
language, also recording the arbitrary-cardinality formula
and its named-zero form.

\begin{theorem}[A subdirect representation for Clifford semirings]
\label{thm:clifford-natural}
Let $S$ be a Clifford inverse semigroup with a zero,
at least one nonzero element, and a natural meet operation.
Equip $S$ with this meet as addition.
For each nonzero idempotent $e$, let $G_e$ be its
maximal subgroup, and put
\[
 G=\prod_{e\in E(S)\setminus\{0\}}G_e.
\]
Then $S$ is a subdirect product of the flat semirings
$\flat(G_e)$, and
\[
 \mathsf{V}(S)=\mathsf{V}(\flat(G)).
\]
These assertions hold in both stated signatures, with
no finiteness assumption.
\end{theorem}

\begin{proof}
Idempotents in a Clifford inverse semigroup are central,
and
\[
 d(s)=r(s),\qquad d(st)=d(s)d(t).
\]
For each nonzero idempotent $e$, define
\[
 \pi_e(s)=
 \begin{cases}
 es\in G_e,&e\leq_{\rm nat} d(s),\\
 0,&e\not\leq_{\rm nat} d(s).
 \end{cases}
\]
Here the zero in the codomain is the new element of
$\flat(G_e)$.  Since $e\leq_{\rm nat} d(st)$ holds exactly
when $e\leq_{\rm nat} d(s)$ and $e\leq_{\rm nat} d(t)$, centrality gives
\[
 (es)(et)=est
\]
in the nonzero case, and both sides of the required
multiplicativity equation are zero in every other case.
Thus $\pi_e$ preserves multiplication and zero.

Let $u=s\wedge t$.
If $\pi_e(s)=\pi_e(t)=v\ne0$, then
$v=es=et$ is a common lower bound of $s,t$.
Consequently $v\leq_{\rm nat} u$, so $e=d(v)\leq_{\rm nat} d(u)$ and
$\pi_e(u)=eu=v$.
Conversely, if $\pi_e(u)\ne0$, then $e\leq_{\rm nat} d(u)$;
the inequalities $u\leq_{\rm nat} s,t$ imply
\[
 eu=es=et\ne0.
\]
Therefore, if the two images are different or at least
one is zero, the image of their meet is zero.
This is exactly the flat addition rule in the codomain.
Hence every $\pi_e$ is a semiring homomorphism.
It is onto, because its restriction to $G_e$ is the
identity and it maps the zero of $S$ to zero.

The family separates elements.
A nonzero $s$ is separated from zero by
$\pi_{d(s)}$.
If nonzero $s,t$ have different domains, one of their
domain idempotents is not below the other; the projection
indexed by that idempotent sends one element to a nonzero
value and the other to zero.
If their domains agree, say at $e$, then
$\pi_e(s)=s$ and $\pi_e(t)=t$.
Thus $S$ embeds subdirectly in
$\prod_{e\ne0}\flat(G_e)$.
The surjections give the reverse variety inclusion.
Finally the zero-smash construction and the coordinate
group embeddings give
\[
 \bigvee_{e\ne0}\mathsf{V}(\flat(G_e))
   =\mathsf{V}\left(\flat\left(\prod_{e\ne0}G_e\right)\right),
\]
as in Theorem~\ref{thm:orthogonal}.
This proves the exact formula.
\end{proof}

\begin{remark}[Adjoining an absorbing zero]
\label{rem:clifford-nozero}
If a non-singleton Clifford inverse semigroup $S$ has
natural meets but no zero, the same variety formula holds
in $\{+,\cdot\}$, with the product over all $e\in E(S)$.
Indeed, $S$ has at least two idempotents: otherwise it
would be a group, whose equality natural order cannot
have all binary meets unless it is a singleton.
Distinct idempotents yield comparable distinct idempotents
$p<_{\rm nat}q$, by taking their product if necessary.
Both operations on $\{p,q\}$ are meet.
Assigning an exclusively occurring variable the value
$p$ and all other variables the value $q$ proves that
every valid identity of $S$ has matching variable content.
Adjoining a fresh absorbing zero therefore preserves all
its binary identities.
The resulting $S^0$ is still Clifford and has natural
meets, so Theorem~\ref{thm:clifford-natural} applies.
Since $\mathsf{V}(S)=\mathsf{V}(S^0)$, it gives the asserted formula.
The singleton algebra is treated separately.
\end{remark}

\begin{corollary}[All finite natural Clifford semirings]
\label{cor:finite-clifford-natural}
\textup{\cite[Theorem 8]{Jackson19}.}
Let $S$ be a finite Clifford inverse semigroup whose
natural order has binary meets, and use these meets as
addition.  Then $S$ is finitely based if and only if
every Sylow subgroup of every maximal subgroup is abelian.
This holds in the binary language and in the language
with its named absorbing zero.
\end{corollary}

\begin{proof}
The singleton case is immediate.
The meet of all elements of the finite set $S$ is a
least element $z$.  Since $z$ is below an idempotent,
it is itself idempotent.  For every $s$, the natural
order relation $z\leq_{\rm nat} s$ gives
\[
 z=s(z^{-1}z)=(zz^{-1})s,
\]
so $sz=zs=z$.  Thus $z$ is a multiplicative zero
and also absorbs meet.
Apply Theorem~\ref{thm:clifford-natural} and the finite
flat-group criterion \eqref{eq:group-criterion}.
The product-group Sylow condition is equivalent to the
condition in each factor, as in
Theorem~\ref{thm:finite-regular}.
\end{proof}

Theorem~\ref{thm:clifford-natural} also reduces the
arbitrary-cardinality natural Clifford problem exactly
to one flat-group problem.
The complete reduction to finite exponent and a finite
quasi-equational basis is stated in
Corollary~\ref{cor:clifford-all-basis} below.

\begin{corollary}\label{cor:clifford-all-basis}
Let $S\ne\{0\}$ be a natural Clifford semiring with zero,
of arbitrary cardinality, and put
\[
 G=\prod_{e\in E(S)\setminus\{0\}}G_e.
\]
Then $S$ is finitely based if and only if its maximal
subgroups have a common finite exponent and $G$
has a finitely axiomatizable group quasi-equational theory.
If there is no common finite exponent, $S$ has no
bounded-variable identity basis.
In the binary language the zero assumption can be removed
for non-singleton $S$, using the product over all
idempotents.
\end{corollary}

\begin{proof}
Theorem~\ref{thm:clifford-natural} identifies
$\mathsf{V}(S)$ with $\mathsf{V}(\flat(G))$.
The product group has finite exponent precisely when
its factors have a common finite exponent.
Apply Theorem~\ref{thm:flat-group-all} and
Remark~\ref{rem:clifford-nozero}.
\end{proof}

We next consider locally Clifford inverse semirings.
Recall that two idempotents $f,g$ belong to the same
Green $\mathcal D$-class if there is an element $a$
with $d(a)=f,r(a)=g$.
For a nonzero $\mathcal D$-class $D$, let
$\mathcal{B}_D$ be its connected groupoid with a new
zero: multiply arrows as in $S$ when their middle
idempotents match, and give zero otherwise.
The groupoid coordinates in the proof of
Proposition~\ref{prop:regular-structure} identify
$\mathcal{B}_D$ with $B_{I_D}(G_D)$, where
$I_D=E(S)\cap D$ and $G_D$ is one maximal subgroup
in $D$.
Equip it with flat addition.
The construction uses the nonzero $\mathcal D$-classes
and applies to semigroups of arbitrary cardinality.

\begin{theorem}[Subdirect Brandt representation]
\label{thm:local-clifford-representation}
Let $S\ne\{0\}$ be a locally Clifford inverse semigroup
with zero and natural meets.  Then $S$, with natural
meet addition, is a subdirect product of the flat
Brandt semirings $\mathcal{B}_D$ over its nonzero
$\mathcal D$-classes.  In particular
\begin{equation}\label{eq:orthogonalization}
 \mathsf{V}(S)=\bigvee_{D\ne\{0\}}\mathsf{V}(\mathcal{B}_D)
      =\mathsf{V}\left(\bigoplus_{D\ne\{0\}}^{0}\mathcal{B}_D\right).
\end{equation}
All maps preserve zero.
If $S$ is finite, its orthogonal sum on the right has
exactly $|S|$ elements.
\end{theorem}

\begin{proof}
First, below any idempotent $e$ there is at most one
idempotent belonging to a specified $\mathcal D$-class.
Otherwise choose distinct $f,g\leq_{\rm nat} e$ in that class
and an arrow $a$ with $d(a)=f,r(a)=g$.
It belongs to $eSe$, but cannot lie in any subgroup,
since its domain and range differ.
This contradicts the Clifford property of $eSe$.

For a nonzero class $D$, define
\[
 \phi_D(s)=
 \begin{cases}
 sf,&\text{if }f\in I_D\text{ and }f\leq_{\rm nat} d(s),\\
 0,&\text{if no such }f\text{ exists}.
 \end{cases}
\]
The preceding uniqueness makes this well-defined.
In the nonzero case $d(sf)=f$, so $sf$ belongs to
$D$, as required.
Every member of $D$ is its own image, and zero maps
to zero; hence $\phi_D$ is onto.

The map $\phi_D$ preserves multiplication.
If $\phi_D(st)\ne0$, choose $h\in I_D$ with
\[
 h\leq_{\rm nat} d(st)=t^{-1}d(s)t.
\]
Then $h\leq_{\rm nat} d(t)$, so $\phi_D(t)=th$.
Its range $g=tht^{-1}$ is an idempotent in $D$
below $d(s)$, giving $\phi_D(s)=sg$.
These two images are composable and their product is
\[
 (sg)(th)=sth=\phi_D(st).
\]
Conversely, suppose $\phi_D(s)=sf$ and
$\phi_D(t)=th$ are nonzero and composable.
The matching condition is $f=tht^{-1}$.
Since $f\leq_{\rm nat} d(s)$, it follows that
$h\leq_{\rm nat} t^{-1}d(s)t=d(st)$; thus
$\phi_D(st)=sth=(sf)(th)$.
We have proved that a product image is nonzero exactly
when the two images have a nonzero product in
$\mathcal{B}_D$, and that their values agree.

Now let $u=s\wedge t$.
If $\phi_D(s)=\phi_D(t)=a\ne0$, then
$a\leq_{\rm nat} s,t$, so $a\leq_{\rm nat} u$.
Its domain lies in $I_D$ below $d(u)$, and the
uniqueness just proved forces $\phi_D(u)=a$.
Conversely, if $\phi_D(u)=a\ne0$, the restrictions
$a\leq_{\rm nat} u\leq_{\rm nat} s,t$ imply
$\phi_D(s)=\phi_D(t)=a$.
This is precisely the flat addition rule, so
$\phi_D$ preserves meet.

Finally suppose all projections of $s,t$ agree.
If $s\ne0$, choose its class $D$.
Then $\phi_D(s)=s=\phi_D(t)$, which implies
$s\leq_{\rm nat} t$.
Interchanging $s,t$ gives the reverse inequality
when both are nonzero.  A nonzero element is also
separated from zero by its own class.
Thus the projections separate elements and give the
claimed subdirect embedding.
Surjectivity gives the reverse variety inclusion;
the orthogonal-sum equality follows from
Theorem~\ref{thm:orthogonal}.
The nonzero carriers $D$ partition $S\setminus\{0\}$,
so the orthogonal sum has the same cardinality in the
finite case.
\end{proof}

\begin{corollary}\label{cor:local-clifford-all}
Let $S$ be a non-Clifford locally Clifford inverse
semigroup with zero and natural meet addition, of any
cardinality.  Then $S$ is finitely based if and only
if it is combinatorial.
In every negative case its identities have no basis
with a uniform variable bound.
For a Clifford $S$, the exact flat-group reduction
of Theorem~\ref{thm:clifford-natural} applies instead.
\end{corollary}

\begin{proof}
A non-Clifford $S$ has a class $D$ containing at
least two idempotents.
Equations \eqref{eq:orthogonalization} and
\eqref{eq:orthogonal-normal-form} identify its variety
with that of $B_2(\prod_DG_D)$.
This product group is trivial exactly in the
combinatorial case.
Apply Corollary~\ref{cor:brandt-classification}.
\end{proof}

The zero hypothesis may again be removed in the binary
language for non-singleton semigroups.
The variable-content argument of
Remark~\ref{rem:clifford-nozero} needs only two comparable
distinct idempotents; these exist in every non-singleton
inverse semigroup with natural meets.
Adjoining a fresh zero preserves local Cliffordness,
natural meets, and all binary identities.

\subsection{Regular semirings}
\label{sec:regular-identity-characterization}

The natural inverse-semigroup hypotheses can be recovered from
regularity and a single identity in the binary semiring language.
As before, regularity refers only to the multiplicative reduct.
Write $\mathcal F$ for the variety generated by all flat
ai-semirings, of arbitrary cardinality.
The identity in the following theorem is an instance of the
flat identity scheme in \cite[Lemma 2.1]{JRZ}.

\begin{theorem}\label{thm:regular-flat-identity}
Let $S$ be an ai-semiring whose multiplicative reduct is regular.
The following conditions are equivalent:
\begin{enumerate}[$(\rm i)$]
\item $S$ satisfies
\begin{equation}\label{eq:regular-flat-identity}
 u+yuz+yvz\approx v+yuz+yvz;
\end{equation}
\item the multiplicative reduct of $S$ is a locally Clifford
inverse semigroup, and addition is meet in its natural order;
\item $S\in\mathcal F$.
\end{enumerate}
The equivalence holds for semirings of arbitrary cardinality
in the binary language and in its expansion by a named
common absorbing zero.
\end{theorem}

\begin{proof}
We first prove (i)$\Rightarrow$(ii) in six steps.  The
inverse notation in the proof refers to semigroup inverses
in the multiplicative reduct.

\emph{Step 1: the multiplication is inverse.}
If idempotents $e,f$ belong to the same Green $\mathcal R$-class
or to the same $\mathcal L$-class, then
$efe=e$ and $fef=f$.
Use $(u,v,y,z)=(e,f,e,e)$ in
\eqref{eq:regular-flat-identity} to obtain $e=e+f$, and use
$(e,f,f,f)$ to obtain $e+f=f$.  Hence $e=f$.
Thus each $\mathcal R$-class and each $\mathcal L$-class
contains at most one idempotent.
Regularity supplies an inverse of every element: if $aba=a$,
replace $b$ by $bab$.
If $b,c$ are inverses of $a$, the idempotents $ab,ac$
belong to the $\mathcal R$-class of $a$, and $ba,ca$
belong to its $\mathcal L$-class.  Consequently
$ab=ac$ and $ba=ca$, and
\[
 b=bab=(ca)b=c(ab)=c(ac)=cac=c.
\]
The inverse is unique, as required.

\emph{Step 2: addition of idempotents is their product.}
Let $e,f$ be idempotents; they commute by Step 1.
Substitution $(u,v,y,z)=(e,ef,f,f)$ in
\eqref{eq:regular-flat-identity} gives $e+ef=ef$.
Interchanging $e,f$ gives $f+ef=ef$.
Put $t=e+f$.  Distributivity yields
\[
 et=ft=te=tf=ef,\qquad t^2=ef.
\]
From $et=ft$, multiplication on the right by
$t^{-1}$ gives $er(t)=fr(t)$.  Since idempotents
commute, $r(t)e=r(t)f$, and hence
\[
 t=r(t)t=r(t)(e+f)=r(t)e+r(t)f=r(t)e.
\]
Thus $t$ is idempotent, so $t=t^2=ef$.  We have proved
\begin{equation}\label{eq:idempotent-natural-addition}
 e+f=ef\qquad(e,f\in E(S)).
\end{equation}

\emph{Step 3: the inverse semigroup is locally Clifford.}
Take $a\in eSe$, where $e$ is idempotent, and put
$g=d(a)$, $f=r(a)$.  Both $f,g$ lie below $e$.
For $(u,v,y,z)=(g,e,a,a^{-1})$, both sandwiched
products in \eqref{eq:regular-flat-identity} equal $f$.
Using \eqref{eq:idempotent-natural-addition}, the identity
therefore gives $gf=f$.
Applying the same argument to $a^{-1}$ gives $gf=g$.
Thus $d(a)=r(a)$, so $a$ belongs to a subgroup.
Every local monoid $eSe$ is Clifford.

\emph{Step 4: distinct elements of one subgroup are not
comparable under addition.}
Suppose $g,h\in G_f$ and $g+h=h$, where $f$ is
the identity of this subgroup.  Put $q=g^{-1}h$.
Left multiplication by $g^{-1}$, and then by $q^{-1}$,
gives
\[
 f+q=q,\qquad q^{-1}+f=f.
\]
Now substitute $(u,v,y,z)=(f,q,q^{-1},f)$ in
\eqref{eq:regular-flat-identity}.  Its two sides reduce,
respectively, to $f$ and $q$.  Hence $q=f$ and
$g=h$.

\emph{Step 5: a sum is a natural lower bound.}
Let $t=a+b$ and $e=d(t)$, and define
\[
 x=t^{-1}ae,\qquad y=t^{-1}be.
\]
Then $x,y\in eSe$, and $x+y=t^{-1}te=e$.
In particular $x+e=e$.
By Step 3, $x$ lies in a subgroup with identity
$h=d(x)=r(x)$.  Its inverse also belongs to $eSe$.
Left multiplication of $x+e=e$ by $x^{-1}$ yields
$h+x^{-1}=x^{-1}$.
Step 4 now gives $x^{-1}=h$, so $x=h$ is idempotent.
The same argument makes $y$ idempotent.
Step 2 gives $xy=x+y=e$, while $x,y\leq_{\rm nat} e$ in the
natural order.  It follows that $x=y=e$.
Multiplying $t^{-1}ae=e$ on the left by $t$ gives
\[
 t=tt^{-1}ae=r(t)a\,d(t)\leq_{\rm nat} a.
\]
The last inequality is the natural order: multiplication
by idempotents on either side restricts an inverse-semigroup
element.  Similarly $t\leq_{\rm nat} b$.

\emph{Step 6: this lower bound is the greatest one.}
If $c\leq_{\rm nat} a$ in the natural order, then
\[
 r(c)c\,d(c)=r(c)a\,d(c)=c.
\]
Substitution $(u,v,y,z)=(c,a,r(c),d(c))$ in
\eqref{eq:regular-flat-identity} gives $c+a=c$.
If $c\leq_{\rm nat} a,b$, we therefore have
$c+t=c+(a+b)=c$.
By Step 5, applied to the pair $c,t$, its sum $c$
lies naturally below $t$.  Thus $t=a\wedge b$,
completing (i)$\Rightarrow$(ii).

For (ii)$\Rightarrow$(iii), apply the subdirect Brandt
representation in Theorem~\ref{thm:local-clifford-representation}.
If a zero is absent in the binary language, the zero-adjunction
argument following Corollary~\ref{cor:local-clifford-all}
preserves all identities and reduces to that theorem.
The one-element case is immediate.

Finally every flat semiring satisfies
\eqref{eq:regular-flat-identity}.  If $u=v$, the two
sides agree.  Otherwise $u+v=0$, so
$yuz+yvz=y(u+v)z=0$, which absorbs the remaining sum.
The identity is preserved under homomorphic images,
subalgebras and products.  This proves (iii)$\Rightarrow$(i).
\end{proof}

The ambient variety $\mathcal F$ has the finite basis
from \cite[Lemma 2.1]{JRZ}: besides the ai-semiring axioms,
use the sixteen instances of
\[
 x_1ux_2+y_1uy_2+y_1vy_2
 \approx x_1vx_2+y_1uy_2+y_1vy_2,
\]
where each of $x_1,x_2,y_1,y_2$ may independently be absent.
Theorem~\ref{thm:regular-flat-identity} shows that, among
regular ai-semirings, the single four-variable identity
\eqref{eq:regular-flat-identity} already characterizes
membership in this finitely based ambient variety.
It also recovers the inverse structure and the natural addition.
In the named-zero language include the common-zero absorption
axioms.  The same flat-variety description holds, since a
common absorbing zero in a nontrivial flat quotient is its
unique existing zero.

\subsection{Finite-basis classification}
\label{sec:s7-consequences}

The following divisor lemma connects local Clifford
structure with Brandt monoids.

\begin{lemma}\label{lem:brandt-monoid-divisor}
Let $S$ be a finite inverse semigroup with natural
meet addition.  If $S$ is not locally Clifford,
then the naturally semilattice-ordered Brandt monoid $B_2^1$
is a subsemiring of a homomorphic image of a
subsemiring of $S$.
\end{lemma}

\begin{proof}
The least element of $S$ is an absorbing zero, by
the argument in Corollary~\ref{cor:finite-clifford-natural}.
Choose $e$ such that the finite inverse monoid
$M=eSe$ is not Clifford.
It is closed under natural meet: a restriction of
an element of $eSe$ again has domain and range
below $e$.
Choose in $M$ an element with different domain and
range, and let $D$ be its $\mathcal D$-class.
Thus $D$ has at least two idempotents and is not the
class of units of $M$.

We use two elementary finite-semigroup facts.
First, $\mathcal D=\mathcal J$ in a finite semigroup.
Consequently, for $J=MDM$, the set $I=J\setminus D$
is an ideal of $M$.
Second, naturally comparable elements in one
$\mathcal D$-class of a finite inverse semigroup
are equal.
For the latter assertion it suffices first to consider
idempotents $f\leq_{\rm nat} g$ in one class.
An arrow between them gives mutually inverse bijections
of the principal right ideals $fM$ and $gM$.
But $fM\subseteq gM$, and if $f\ne g$, then
$g\notin fM$, contradicting equality of their finite
cardinalities.
For $a\leq_{\rm nat} b$ in one class, apply this to
$d(a)\leq_{\rm nat} d(b)$; their equality gives $a=b$.

Put $R=J\cup\{e\}$.  Here $e\notin J$, since the
chosen class is not the class of units.
The set $R$ is closed under multiplication.
It is closed under meet as well: an element below a
member of $J$ belongs to that ideal.
The same observation shows that $i\wedge r\in I$
for $i\in I,r\in R$.
Thus collapsing $I$ is a semiring Rees quotient.

The multiplicative quotient $J/I=D\cup\{0\}$ is
the Brandt principal factor $B_n(G)$, with $n\geq2$.
One may see its coordinates directly from the arrows
between the idempotents of $D$: a nonmatching product
must lie below $D$, by the comparability fact, and
matching products are groupoid composition.
The image of $e$ adjoins an identity.
Distinct elements of $D$ have their meet in $I$,
again by comparability.
Also $e\wedge a=a$ for idempotent $a\in D$;
for nonidempotent $a\in D$, this meet lies in $I$,
because an element below $e$ is idempotent.
These are exactly the natural meet rules in
$B_n^1(G)$.  Hence
\[
 R/I\cong (B_n^1(G),+_{\rm nat},\cdot).
\]
Choose two indices, the identity group labels, the
zero, and the adjoined identity.
They form the required subsemiring $B_2^1$.
All constructions preserve the absorbing zero.
\end{proof}

The standard semiring $S_7$ has three elements $1,a,0$,
flat addition, a multiplicative identity $1$, and $a^2=0$.
For its operations and earlier nonfinite-basis results see
\cite{JRZ}.
A finite ai-semiring $A$ is \emph{strongly nonfinitely based}
(SNFB) if every finite ai-semiring $T$ with
$A\in\mathsf{V}(T)$ is nonfinitely based.

\begin{lemma}\label{lem:s7-snfb}
The three-element semiring $S_7$ is strongly nonfinitely
based in the binary ai-semiring language.  Equivalently,
\[
 T\text{ finite},\quad S_7\in\mathsf{V}(T)
 \quad\Longrightarrow\quad T\text{ is nonfinitely based}.
\]
\end{lemma}

\begin{lemma}[Transfer and the absorbing constant]
\label{lem:s7-transfer}
If a finite ai-semiring $A$ satisfies $S_7\in\mathsf{V}(A)$,
then $A$ is itself SNFB.  Lemma~\ref{lem:s7-snfb} also holds
after naming the existing absorbing zero of $S_7$,
among ai-semirings with a named absorbing zero.
More generally, if $\mathsf{V}(A)=\mathsf{V}(T)$ for a finite $T$,
then $S_7\in\mathsf{V}(A)$ implies that $A$ is nonfinitely
based, whether or not $A$ is finite.
\end{lemma}

\begin{proof}
If $A\in\mathsf{V}(T)$, transitivity gives $S_7\in\mathsf{V}(T)$;
apply Lemma~\ref{lem:s7-snfb} when $T$ is finite.
The assertion about equal varieties is the same argument.

For the signature assertion let finite $(T;0)$ generate
a variety containing $(S_7;0)$, where the notation means
that the existing zero is named.  Its binary reduct already
has $S_7$ in its variety.  Lemma~\ref{lem:constant}
equates finite basability in the two signatures, and
Lemma~\ref{lem:s7-snfb} gives the result.
\end{proof}

\begin{lemma}[A separating sandwich identity]
\label{lem:s7-sandwich}
Let $r\geq2$, and let the exponent of $G$ divide $r$.
Every $B_I(G)$, $I\ne\varnothing$, satisfies
\begin{equation}\label{eq:s7-sandwich}
 x^r yx^r\approx x^r y^{r+1}x^r.
\end{equation}
The semiring $S_7$ fails this identity.
\end{lemma}

\begin{proof}
In a Brandt semiring, $x^r$ is either zero or a diagonal
matrix idempotent $e_i=(i,1,i)$.  In the latter case,
if $y$ belongs to the group at $i$, then
$y^{r+1}=y$.  If $y$ does not belong to that group,
both sandwiched products are zero: a different diagonal
corner stays different, and a nonzero off-diagonal element
has square zero.  This proves the identity, including all
zero assignments.  In $S_7$, take $x=1,y=a$; the two
sides are $a$ and $0$.
\end{proof}

\begin{theorem}[Varieties containing $S_7$]
\label{thm:s7-finite-detection}
For a finite natural inverse semiring $S$,
\[
 S_7\in\mathsf{V}(S)
 \quad\Longleftrightarrow\quad
 S\text{ is not locally Clifford}.
\]
Every finite natural inverse semiring that is
not locally Clifford is SNFB and hence nonfinitely based.
\end{theorem}

\begin{proof}
If $S$ is not locally Clifford, the divisor in
Lemma~\ref{lem:brandt-monoid-divisor} contains
\[
 \{1,e_{12},0\}\cong S_7:
 \qquad e_{12}^2=0,\quad1+e_{12}=0.
\]
Thus $S_7\in\mathsf{V}(S)$, with the zero preserved.
If $S\ne\{0\}$ is locally Clifford, choose a common
multiple $r\geq2$ of the exponents of its finitely many
maximal subgroups.  Its Brandt factors in
Theorem~\ref{thm:local-clifford-representation} satisfy
\eqref{eq:s7-sandwich}, and therefore so does $S$.
The lemma excludes $S_7$ from its variety.
The zero-only case is immediate.
The SNFB assertion follows from Lemma~\ref{lem:s7-transfer}.
\end{proof}

\begin{theorem}[All finite natural inverse semirings]
\label{thm:finite-natural-classification}
Let $S$ be a finite inverse semigroup whose natural
order has binary meets, and use them as addition.
In either $\{+,\cdot\}$ or $\{+,\cdot,0\}$,
$S$ is finitely based if and only if at least one
of the following holds:
\begin{enumerate}[$(\rm i)$]
\item $S$ is locally Clifford and combinatorial;
\item $S$ is Clifford and all Sylow subgroups of
every maximal subgroup are abelian.
\end{enumerate}
\end{theorem}

\begin{proof}
The singleton case is immediate; otherwise $S$
has its natural absorbing zero.
If $S$ is locally Clifford, its orthogonalization
in Theorem~\ref{thm:local-clifford-representation}
is a finite regular flat semiring with exactly the
same identities.  Its group factors and its classes
with two or more idempotents retain precisely those
features of $S$.  Theorem~\ref{thm:finite-regular}
therefore gives alternatives (i) and (ii).
If $S$ is not locally Clifford,
Theorem~\ref{thm:s7-finite-detection} and
Lemma~\ref{lem:s7-snfb} give nonfinite basability directly.
\end{proof}

\begin{corollary}[A membership characterization]
\label{cor:finite-natural-brandt-variety}
For a finite natural inverse semiring $S$,
\[
 S\in\mathcal{B}\quad\Longleftrightarrow\quad
 S\text{ is locally Clifford}.
\]
\end{corollary}

\begin{proof}
The subdirect representation proves membership in
the locally Clifford case.
Otherwise the divisor lemma would place $B_2^1$
in $\mathcal{B}$, contrary to the explicit failure of
the prism identities in Section~\ref{sec:boundary}.
\end{proof}

\begin{theorem}\label{thm:s7-exponent}
For every group $G$ and every nonempty index set $I$,
the following are equivalent:
\begin{enumerate}[$(\rm i)$]
\item $S_7\in\mathsf{V}(\flat(G))$;
\item $S_7\in\mathsf{V}(B_I(G))$;
\item there is no positive integer $r$ such that
$g^r=1$ for all $g\in G$.
\end{enumerate}
\end{theorem}

\begin{proof}
A diagonal corner gives (i)$\Rightarrow$(ii).
Lemma~\ref{lem:s7-sandwich} gives (ii)$\Rightarrow$(iii),
taking a multiple $r\geq2$ of a finite exponent if necessary.
For (iii)$\Rightarrow$(i), choose for every $m\geq1$
an element $g_m\in G$ with $g_m^m\ne1$.
Then $g=(g_m)_{m\geq1}\in G^{\mathbb N}$ has infinite
order, because its $d$-th coordinate rules out $g^d=1$.
Lemma~\ref{lem:smash} gives
\[
 \flat(G^{\mathbb N})\in\mathsf{V}(\flat(G)).
\]
In this semiring take the subsemiring
$T=\{0\}\cup\{g^j:j\geq0\}$ and its ideal
$J=\{0\}\cup\{g^j:j\geq2\}$.
All powers displayed are distinct.  The semiring Rees
quotient $T/J$ has the two nonzero elements $1,g$,
with $g^2=0$, and flat addition.  It is $S_7$.
These constructions all preserve the absorbing zero.
\end{proof}

\begin{corollary}\label{cor:s7-local-exponent}
Let $S\ne\{0\}$ be a locally Clifford inverse
semigroup with zero and natural meet addition, of any
cardinality.  Then $S_7\in\mathsf{V}(S)$ if and only if
its maximal subgroups have no common finite exponent.
\end{corollary}

\begin{proof}
Put $G=\prod_DG_D$, taking one maximal subgroup in
each nonzero $\mathcal D$-class.  Theorems
\ref{thm:local-clifford-representation} and
\ref{thm:orthogonal} identify $\mathsf{V}(S)$ with either
$\mathsf{V}(\flat(G))$ or $\mathsf{V}(B_2(G))$.
The group $G$ has exponent dividing $r$ exactly
when every factor does.  Apply Theorem~\ref{thm:s7-exponent}.
\end{proof}

A common finite exponent is a single positive integer
divisible by the exponent of every maximal subgroup.
In particular, every finite Brandt variety excludes
$S_7$, including every $\mathsf{V}(B_2(C_p))$.
Conversely, $S_7\in\mathsf{V}(B_2(\mathbb Z))\subseteq\mathcal{B}$.

\begin{lemma}[Variable bounds in locally finite varieties]
\label{lem:locally-finite-rank}
In a finite signature, a locally finite variety is finitely
based if and only if it has an identity basis with a uniform
finite bound on the number of variables.
Consequently every finite nonfinitely based algebra has
no bounded-variable identity basis.
\end{lemma}

\begin{proof}
Only the reverse direction needs proof.  Suppose a basis
uses at most $k$ variables.  The relatively free algebra
$F_{\mathcal V}(k)$ is finite.  Choose a representing
$k$-variable term $\mathbf{t}_a$ for each of its elements $a$.
For every basic operation $f$ and every tuple of elements,
take the valid identity
\[
 f(\mathbf{t}_{a_1},\ldots,\mathbf{t}_{a_m})\approx \mathbf{t}_{f(a_1,\ldots,a_m)},
\]
together with $x_i\approx \mathbf{t}_{[x_i]}$; nullary operations,
if present, are included in the same way.
This is a finite family of identities valid in $\mathcal V$.
Induction on a term reduces every $k$-variable term to
the representative of its value in the free algebra.
Thus this finite family entails every valid $k$-variable
identity and hence the proposed bounded-variable basis.
It is a finite basis for $\mathcal V$.
\end{proof}

Consequently every finite target covered by
Lemma~\ref{lem:s7-snfb} has no bounded-variable identity
basis.

\begin{theorem}\label{thm:three-obstructions}
Let $S$ be a finite natural inverse semiring.
Then $S$ is nonfinitely
based if and only if at least one of the following holds:
\begin{enumerate}[$(\rm i)$]
\item $S_7\in\mathsf{V}(S)$;
\item $B_2(C_p)\in\mathsf{V}(S)$ for some prime $p$;
\item $\flat(P)\in\mathsf{V}(S)$ for some finite
nonabelian $p$-group $P$.
\end{enumerate}
Each obstruction can occur while the other two are absent,
even within the class of finite natural inverse semirings.
Thus none of the three alternatives can be omitted.
\end{theorem}

\begin{proof}
If $S$ is not locally Clifford, (i) follows from
Theorem~\ref{thm:s7-finite-detection}.
If $S$ is locally Clifford but not Clifford and is
nonfinitely based, it has a nontrivial maximal subgroup.
The orthogonalization and product-group formula put
$B_2(C_p)$ in its variety for a prime divisor $p$
of the order of such a subgroup.
If $S$ is Clifford and nonfinitely based, its finite
flat-group reduction has a nonabelian Sylow subgroup;
the corresponding flat group is a divisor in its variety.
This proves necessity using
Theorem~\ref{thm:finite-natural-classification}.

Condition (i) implies nonfinite basability by Lemma~\ref{lem:s7-snfb}.  Condition (iii) does so by
\cite[Theorem 6.1]{JRZ}, which states that every finite
ai-semiring whose variety contains a finite ai-semiring
with a nonabelian nilpotent subgroup is nonfinitely based.
For (ii), if $S$ were finitely based, its two
possible structures in
Theorem~\ref{thm:finite-natural-classification} would
both exclude $B_2(C_p)$: a locally Clifford
combinatorial $S$ satisfies $x^2\approx x^3$,
while a finite Clifford $S$ satisfies
$x^{r+1}\approx x$ for a common group exponent $r$.
The first fails in its nontrivial group corner and the
second fails at an off-diagonal element of $B_2(C_p)$.

To show that no alternative can be omitted, use,
respectively, $B_2^1$,
$B_2(C_p)$, and $\flat(P)$.
The first satisfies $x^2\approx x^3$, excluding
both types of nontrivial-group obstruction.
The second excludes $S_7$ by the sandwich identity;
it also satisfies
\[
 x^{p+2}\approx x^2,\qquad
 x^{p+1}y^{p+1}\approx y^{p+1}x^{p+1}.
\]
The latter identity holds because each $(p+1)$-st
power is zero or stays in an abelian diagonal group.
Any group whose flat extension belongs to this variety
must therefore have exponent dividing $p$ and be abelian.
Finally $\flat(P)$ satisfies $x^{r+1}\approx x$,
where $r$ is its exponent; this excludes both $S_7$
and every $B_2(C_q)$.
\end{proof}

\begin{remark}
The three obstructions cannot be combined arbitrarily:
\textup{(i)} and \textup{(iii)} together imply
\textup{(ii)}.  Indeed, \textup{(i)} gives a natural
$B_2^1$ divisor and hence $B_2\in\mathsf{V}(S)$.
Condition \textup{(iii)} gives $\flat(C_p)\in\mathsf{V}(S)$.
The Brandt join formula then puts $B_2(C_p)$ in
$\mathsf{V}(S)$.
\end{remark}

\begin{corollary}[Transfer within the finite natural inverse class]
\label{cor:relative-natural-transfer}
Let $A,B$ be finite natural inverse semirings
with $A\in\mathsf{V}(B)$.  If $A$ is nonfinitely based,
then $B$ is nonfinitely based.
Equivalently, every finite natural inverse member of the
variety generated by a finitely based finite natural
inverse semiring is finitely based.
\end{corollary}

\begin{proof}
One of the three obstructions in
Theorem~\ref{thm:three-obstructions} belongs to $\mathsf{V}(A)$,
and therefore to $\mathsf{V}(B)$.  Apply the same theorem to
$B$.
\end{proof}

\begin{corollary}\label{cor:finite-regular-identity}
Let $S$ be a finite ai-semiring with regular multiplication
satisfying \eqref{eq:regular-flat-identity}.
Then $S$ is finitely based if and only if either
\begin{enumerate}[$(\rm i)$]
\item its multiplicative semigroup is combinatorial; or
\item its multiplicative semigroup is Clifford and every
Sylow subgroup of every maximal subgroup is abelian.
\end{enumerate}
This holds in both signatures considered here.
\end{corollary}

\begin{proof}
Theorem~\ref{thm:regular-flat-identity} makes $S$ a
locally Clifford natural inverse semiring.
Apply Theorem~\ref{thm:finite-natural-classification}.
\end{proof}

\begin{corollary}\label{cor:relative-regular-transfer}
Let $B$ be a finitely based finite regular ai-semiring
satisfying \eqref{eq:regular-flat-identity}.
Every finite member $A\in\mathsf{V}(B)$ whose multiplication
is regular is finitely based.
\end{corollary}

\begin{proof}
The identity passes from $B$ to $A$, so both are
natural inverse semirings by
Theorem~\ref{thm:regular-flat-identity}.
Apply Corollary~\ref{cor:relative-natural-transfer}.
\end{proof}

\begin{proposition}\label{prop:flat-local-s7}
Let $S$ be any finite flat ai-semiring.  If for some
nonzero idempotent $e$, the local monoid $eSe$
has a nonzero nonunit, then $S$ contains $S_7$
as a subsemiring.  Moreover, $S$
is SNFB and hence nonfinitely based.
\end{proposition}

\begin{proof}
The local monoid $M=eSe$ is a flat subsemiring with
unit $e$ and zero $0$.
The structural step is \cite[Proposition 2.4]{JRZ}
applied to $M$; we include its short argument.
It has no nonzero idempotent other than $e$: if
$f^2=f\ne0$, then $ef=ff\ne0$, and zero-cancellation
gives $e=f$.
For a nonzero nonunit $b$, finiteness gives an
idempotent power.  This cannot be $e$, since
$b^m=e$ would make $b$ a unit of $M$.
Thus some power is zero.  If $b^j$ is the last
nonzero power, set $a=b^j$.  Then $a^2=0$,
and $\{e,a,0\}\cong S_7$ with flat addition.
The conclusion follows from Lemma~\ref{lem:s7-transfer}.
\end{proof}

For a nonzero finite flat semiring that itself has a
multiplicative identity, this recovers the known
classification of \cite[Corollary 4.13]{JRZ}:
it is finitely based exactly when it is $\flat(G)$
for a finite group whose Sylow subgroups are all abelian.
Indeed, either every nonzero element is a unit or the
proposition applies.  This local criterion applies to all
finite flat semirings, including those with nonregular
multiplication.

\subsection{Minimal generators}
\label{sec:smallest}
\label{sec:smallest-natural}

\begin{theorem}[The smallest nonfinitely based regular example]
\label{thm:seven}
Among finite flat semirings with regular multiplication,
the smallest nonfinitely based cardinality is seven.
Up to isomorphism there is exactly one such seven-element
semiring:
\[
 R_7=B_2\oplus^{0} \flat(C_2).
\]
Its variety is $\mathsf{V}(B_2(C_2))$.
\end{theorem}

\begin{proof}
The variety assertion follows from
Theorem~\ref{thm:orthogonal}, so $R_7$ is nonfinitely
based.  Its two blocks have four and two nonzero elements,
and share one zero.

In a finite Clifford negative case,
Theorem~\ref{thm:finite-regular} forces a nonabelian Sylow
subgroup.  A group of order $p$ or $p^2$ is abelian,
so a nonabelian $p$-group has order at least $p^3\geq8$.
Such a block already contributes at least nine elements
including zero.
In a non-Clifford negative case a rank-at-least-two block
and a nontrivial group are both required.  If in one block,
they give at least $2^2\cdot2=8$ nonzero elements.
If in different blocks, they give at least $4+2=6$.
Equality at total size seven forces exactly the displayed
two blocks and no others.
\end{proof}

The next result determines the smallest regular flat
generators of every finite cyclic Brandt variety.

\begin{theorem}[Unique minimum cyclic-variety generators]
\label{thm:cyclic-minimum}
Let $r\geq2$ and write
\[
 r=\prod_{i=1}^s q_i,\qquad q_i=p_i^{a_i},
\]
where the primes $p_i$ are distinct and $a_i\geq1$.
Among all finite regular flat semirings generating
$\mathsf{V}(B_2(C_r))$, the minimum size is
\[
 5+\sum_{i=1}^s q_i.
\]
Up to isomorphism the unique minimum generator is
\[
 B_2\oplus^{0}\flat(C_{q_1})\oplus^{0}\cdots
       \oplus^{0}\flat(C_{q_s}).
\]
In particular, for a prime power $r=p^a$ the minimum is
$p^a+5$, uniquely attained by
$B_2\oplus^{0} \flat(C_{p^a})$.
\end{theorem}

\begin{proof}
The displayed semiring has the required variety because
$C_r\cong\prod_i C_{q_i}$ and
Theorem~\ref{thm:orthogonal} applies.
Let a finite regular flat semiring $A$ generate that
same variety, and write its Brandt blocks as
$B_{n_\lambda}(G_\lambda)$.
The target satisfies $x^{r+2}\approx x^2$: on a
diagonal element this is the exponent law, and on an
off-diagonal element both sides are zero.  In a corner
group of $A$, cancellation therefore gives exponent
dividing $r$.

For each $i$, the target fails
$x^{r/p_i+2}\approx x^2$, so $A$ fails it too.
Off-diagonal elements cannot witness this failure; hence
some corner contains an element whose order has full
$p_i$-part $q_i$, and therefore contains an element
of order $q_i$.  Assign requirement $q_i$ to one
such block.  A group supporting a set $Q$ of assigned
requirements has order at least $\prod_{q\in Q}q$,
by Lagrange's theorem.  For two or more distinct prime
powers their product is strictly greater than their sum;
for a singleton it equals that requirement.

There must also be a block with $n_\lambda\geq2$.
Otherwise every element lies in a group of exponent
dividing $r$ or is zero, and $x^{r+1}\approx x$
would hold in $A$, contrary to its failure on an
off-diagonal element of $B_2(C_r)$.
Choose one such block.  It contributes at least four
nonzero elements in addition to the sum of its assigned
requirements, strictly more if any requirement is
assigned there: for a nonempty assignment,
$4\prod Q>4+\sum Q$.
Every other block contributes at least the sum of its
assigned requirements, with strict inequality if it has
none or more than one, or if its rank exceeds one.
Summing gives at least
$4+\sum_iq_i$ nonzero elements.

Equality forces one trivial-group rank-two block with no
assigned requirement, and one rank-one group of order
$q_i$ for each requirement.  Such a group contains an
element of its full order $q_i$, so it is cyclic.  No
additional blocks are possible.  This proves both
minimality and uniqueness.
\end{proof}

\begin{corollary}\label{cor:cyclic-natural-minimum}
For $r=\prod_iq_i\geq2$ as in
Theorem~\ref{thm:cyclic-minimum}, the minimum size of
a finite natural inverse semiring generating
$\mathsf{V}(B_2(C_r))$ is still
\[
 5+\sum_iq_i.
\]
The orthogonalization of every minimum generator is
the unique flat generator in that theorem.
The minimum generator itself need not be unique once
flatness is removed.
\end{corollary}

\begin{proof}
Any such generator belongs to $\mathcal{B}$, so is
locally Clifford by
Corollary~\ref{cor:finite-natural-brandt-variety}.
Its orthogonalization has the same size and variety.
Theorem~\ref{thm:cyclic-minimum} gives the lower bound
and the asserted orthogonalization.
The flat example there attains the bound.

For nonuniqueness take $r=2$, let
$C_2=\{1,g\}$, and consider the following
seven-element subset of
$B_2\times \flat(C_2)$:
\[
 A=\{(e_{ij},1):1\leq i,j\leq2\}
       \cup\{(0,1),(0,g),(0,0)\}.
\]
It is closed under multiplication, inversion, and
coordinatewise natural meet.
Products of matrix-unit pairs that do not match
give $(0,1)$; operations with the last three
elements stay in the displayed set.
Both coordinate projections are onto.
Hence
\[
 \mathsf{V}(A)=\mathsf{V}(B_2)\vee\mathsf{V}(\flat(C_2))
      =\mathsf{V}(B_2(C_2)).
\]
Its zero is $(0,0)$, while
$(e_{11},1)+(e_{22},1)=(0,1)\ne(0,0)$.
Thus $A$ is nonflat and cannot be isomorphic to
the flat minimum $R_7$.
\end{proof}

The next result puts the small examples directly into
the finite-basis classification.  Say that a semiring
is $S_7$-free if $S_7\notin\mathsf{V}(S)$.

\begin{theorem}\label{thm:smallest-natural}
The smallest
nonfinitely based finite natural inverse
semiring has six elements and is uniquely $B_2^1$,
up to isomorphism.
Among the $S_7$-free members of this class the
smallest nonfinitely based size is seven.
There are exactly four isomorphism types at that size,
and all four generate $\mathsf{V}(B_2(C_2))$.
Consequently they are also exactly the minimum finite
natural inverse-semiring generators of that variety.
\end{theorem}

\begin{proof}
An $S_7$-free finite member is locally Clifford by
Theorem~\ref{thm:s7-finite-detection}.
Its orthogonalization has the same size and identities;
Theorem~\ref{thm:seven} gives the lower bound seven,
with orthogonalization necessarily $R_7$ at equality.
A non-locally-Clifford member has $B_2^1$ as a
semiring divisor and therefore has at least six elements.
The six-element $B_2^1$ contains $S_7$, so is
nonfinitely based by Lemma~\ref{lem:s7-snfb}.
If a six-element $S$ has this divisor, restrict the
quotient map to the inverse image of that subsemiring.
The resulting surjection from a subsemiring of $S$
onto $B_2^1$ must be a bijection on all of $S$.
This proves the six-element assertion.

We now describe and classify the seven-element cases.
Put $C_2=\{1,g\}$, and work inside
$P=B_2\times \flat(C_2)$, with coordinatewise
natural meet. We identify $R_7$ with its coordinate-axis
copy. The four subsemirings are
\begin{align*}
 R_7&=\{(e_{ij},0):1\leq i,j\leq2\}
             \cup\{(0,1),(0,g),(0,0)\},\\
 R_t&=\{(e_{11},1),(e_{22},1),(e_{12},t),(e_{21},t)\}
             \cup\{(0,1),(0,g),(0,0)\},
             \quad t\in\{1,g\},\\
 R'_7&=\{(e_{ij},0):1\leq i,j\leq2\}
             \cup\{(e_{11},1),(e_{11},g),(0,0)\}.
\end{align*}
The other three semirings are nonflat.

For closure of $R_t$, matching matrix products have
the displayed labels because $1t=t$ and $t^2=1$.
Nonmatching products have first coordinate zero and
land in its last three elements, as do all products
involving those elements.  Inversion interchanges the
off-diagonal entries, with $t^{-1}=t$.
The meet of distinct matrix elements has first
coordinate zero, so it too is retained.
For $R'_7$, the four matrix-axis elements,
together with $(0,0)$, form $B_2$.
The two additional nonzero elements have the same
first coordinate $e_{11}$, and form a group of
order two; their products with the other elements
stay on the matrix axis.  Their meet is $(e_{11},0)$,
and their meets with any other matrix element are
$(e_{11},0)$ or $(0,0)$, as appropriate.
The axis construction proves closure for $R_7$.
Thus all four are inverse subsemigroups closed under
the ambient natural meet, with that meet as their own
natural meet.  Both coordinate projections are onto.
Hence all four generate
\[
 \mathsf{V}(B_2)\vee\mathsf{V}(\flat(C_2))
 =\mathsf{V}(B_2(C_2)),
\]
so are nonfinitely based and $S_7$-free.

Let $S$ be any seven-element
$S_7$-free nonfinitely based member.
Its orthogonalization is $R_7$, so its nonzero
$\mathcal D$-classes are
\[
 M=\{e,f,a,b\},\qquad H=\{h,g\}\cong C_2,
\]
where $e,f$ are the two matrix idempotents, $a,b$
are inverse connecting arrows, and $h$ is the unit
of $H$.  The two projections of
Theorem~\ref{thm:local-clifford-representation} embed
$S$ into $B_2\times \flat(C_2)$.
The idempotents $e,f$ are incomparable and cannot
both lie below $h$, by local Cliffordness.
If $h\leq_{\rm nat} e$, conjugating by a connecting arrow
gives an idempotent in the class $H$ below $f$.
It must be $h$, the unique idempotent of $H$.
Thus $h\leq_{\rm nat} e$ implies $h\leq_{\rm nat} f$, and conversely.
Up to exchanging $e,f$, the only possible orders are
\begin{enumerate}[$(\rm a)$]
\item $e,f,h$ pairwise incomparable;
\item $h<_{\rm nat}e$ and $h<_{\rm nat}f$;
\item $e<_{\rm nat}h$, with $f$ incomparable to $h$.
\end{enumerate}
In (a) the cross-class restrictions all vanish, so the
projection image is $R_7$.
In (b), the $H$-coordinates of $e,f$ are both
$1$.  The coordinate of $a$ is a value
$t\in\{1,g\}$, and that of $b$ is $t^{-1}=t$.
The matrix coordinates of $h,g$ vanish, giving
exactly $R_1$ or $R_g$.
In (c), all $H$-coordinates of $M$ vanish.
Both $h,g$ project to $e_{11}$, since their
restriction lies in the trivial group there.
This gives exactly $R'_7$.
The injective projections determine the full operations,
so there are no further cases or parameters.

The orders (a), (b), (c) have, respectively, three,
one, and two nonzero minimal idempotents, distinguishing
the three families.  In $R_1$ and $R_g$, the two
off-diagonal elements are the only elements outside
all subgroups.  Their sum is the idempotent $h$
in $R_1$, but the nonidempotent $g$ in $R_g$.
Thus these two are also nonisomorphic.
The final minimum-generator assertion follows from
Corollary~\ref{cor:cyclic-natural-minimum} at $r=2$.
\end{proof}

\newpage
\section{Products and infinite constructions}
\label{sec:products-infinite}

\subsection{Products and adjoining identities}
\label{sec:arbitrary-local-families}
\label{sec:boundary}

\begin{theorem}\label{thm:product-fb}
Let $S_1,\ldots,S_t$, $t\geq1$, be finite
natural inverse semirings.
Their direct product is finitely based if and only if
either
\begin{enumerate}[$(\rm i)$]
\item every $S_i$ is locally Clifford and combinatorial;
or
\item every $S_i$ is Clifford, and every Sylow subgroup
of every maximal subgroup of every $S_i$ is abelian.
\end{enumerate}
The same criterion holds for their orthogonal sum and
for the finite join $\bigvee_i\mathsf{V}(S_i)$.

In particular, if all factors are finitely based, their
product is finitely based exactly when all factors are
combinatorial or all factors are Clifford.
It is nonfinitely based exactly when some factor is
non-Clifford and some other factor has a nontrivial subgroup.
\end{theorem}

\begin{proof}
The product has natural meets coordinatewise.
Cliffordness, local Cliffordness, and being combinatorial
hold in the product precisely when they hold in every
factor; local monoids and maximal subgroups are the
corresponding direct products.  The Sylow subgroups of
a finite product group are the products of its Sylow
subgroups, so their being abelian is also coordinatewise.
Apply Theorem~\ref{thm:finite-natural-classification}.
The product generates the join of the factor varieties.
Their orthogonal sum embeds on the coordinate axes,
because zero absorbs both operations; each summand also
embeds in that sum.  Hence it generates the same join.
If the factors are already finitely based, they are all
locally Clifford and all their maximal subgroups have
abelian Sylow subgroups.  The criterion reduces to the
last assertion.  The two offending features cannot occur
in one factor, since that factor would not be finitely based.
\end{proof}

\begin{corollary}[Finite natural inverse monoids]
\label{cor:inverse-monoid-fb}
A finite inverse monoid with natural meet addition is
finitely based if and only if it is Clifford and every
Sylow subgroup of every maximal subgroup is abelian.
\end{corollary}

\begin{proof}
A monoid is locally Clifford only if it is Clifford,
by taking its global identity as the corner idempotent;
the converse holds as well.
Apply Theorem~\ref{thm:finite-natural-classification}.
\end{proof}

\begin{theorem}\label{thm:arbitrary-local-family}
Let $(S_\lambda)_{\lambda\in\Lambda}$ be a nonempty family
of locally Clifford inverse semigroups with natural meet
addition, with no cardinality restrictions on the family or
its members.  Put
$\mathcal W=\bigvee_\lambda\mathsf{V}(S_\lambda)
             =\mathsf{V}(\prod_\lambda S_\lambda)$.
If every factor is one-element, $\mathcal W$ is trivial.
Otherwise discard one-element factors and, in the binary
language, adjoin a zero to any remaining factor lacking one.
These changes preserve $\mathcal W$.
For these factors put
\[
 G = \prod_{\lambda\in\Lambda}
       \prod_{D\ne\{0\}}G_{\lambda,D},
\]
where one maximal subgroup is chosen in each nonzero
$\mathcal D$-class of $S_\lambda$.
Then
\[
 \mathcal W=\mathsf{V}\left(\bigoplus_\lambda^{0} S_\lambda\right)
 =\begin{cases}
   \mathsf{V}(B_2(G)),&\text{if some factor is not Clifford},\\
   \mathsf{V}(\flat(G)),&\text{if every factor is Clifford}.
  \end{cases}
\]
Consequently:
\begin{enumerate}[$(\rm i)$]
\item If some factor is not Clifford, $\mathcal W$ is
finitely based exactly when every factor is combinatorial.
Every negative case has no bounded-variable identity basis.
\item If every factor is Clifford, $\mathcal W$ is
finitely based exactly when $G$ has finite exponent and
has a finite basis for its group quasi-equational theory.
If $G$ has no finite exponent, $\mathcal W$ has no
bounded-variable identity basis.
\end{enumerate}
The finite-exponent condition means that all maximal subgroups
throughout the entire family share a common finite exponent.
\end{theorem}

\begin{proof}
Apply Theorem~\ref{thm:local-clifford-representation} to
each factor, and take the join of its Brandt-factor varieties.
The arbitrary orthogonal-sum formula in
Theorem~\ref{thm:orthogonal} combines these factors into
the stated Brandt or flat-group variety.
A class with at least two idempotents occurs precisely when
one of the original factors is not Clifford.
The orthogonal sum of the original factors embeds on the
coordinate axes of their direct product, while each factor
embeds in that sum.  It therefore generates the same join.
The preprocessing without a zero uses the identity-preserving
zero-adjunction argument already used above.

For (i), $G$ is trivial precisely when every maximal
subgroup of every factor is trivial.  Apply the all-group
Brandt classification and its bounded-variable conclusion.
For (ii), apply Theorem~\ref{thm:flat-group-all}.
Finally a direct product of groups has exponent dividing
$r$ exactly when each factor has exponent dividing $r$.
\end{proof}

The prism identities fail in $B_2^1$. For every $m$, assign
\[
 a_i=c_i=e_{11},\qquad \bar a_i=e_{12},\qquad
 \bar c_i=e_{21},\qquad d_i=1,
\]
with $b_i=1$ for $i\ne0$ and $b_0=e_{22}$.
Every retained prism word has value $e_{11}$: the
only special lower product is
$e_{12}e_{22}e_{21}=e_{11}$.
The missing word has value
$e_{11}e_{22}e_{11}=0$.
Thus $\mathbf{w}_0\preceq\mathbf{u}_m$ fails for every $m$.

Flat addition on the enlarged multiplicative monoid fails
distributivity: $e_{11}1=e_{11}e_{11}\ne0$ with
$1\ne e_{11}$, contrary to zero-cancellation.

Here $B_I^1(G)$ is equipped with natural meet addition,
so $1+s=s$ for a matrix idempotent $s$, and
$1+s=0$ for a nonidempotent nonzero matrix element.
The global identity is an element, not a named constant
in either signature of this paper.

\begin{theorem}[Brandt monoid reduction]\label{thm:brandt-monoid-join}
For every group $G$ and every $|I|\geq2$,
\begin{equation}\label{eq:brandt-monoid-join}
 \mathsf{V}(B_I^1(G))=\mathsf{V}(B_2^1(G))
            =\mathsf{V}(B_2^1)\vee\mathsf{V}(\flat(G)).
\end{equation}
\end{theorem}

\begin{proof}
First put $A=B_2^1(G)$, and take a nonempty set $K$.
Inside $A^K$, let $J$ be the set of tuples with a
zero coordinate, put
$C=(B_2(G)\setminus\{0\})^K$, and let $u$ be
the tuple of global identities.  The set
$R=J\cup C\cup\{u\}$ is a subsemiring.
Indeed $J$ absorbs both operations; products and sums
of two members of $C$ belong to $C\cup J$.
For $c\in C$, the sum $u+c$ equals $c$ if
all coordinates of $c$ are matrix idempotents, and
otherwise belongs to $J$.
Collapsing $J$ gives
\[
 R/J\cong B_{\{1,2\}^K}^1(G^K).
\]
The identification collects the left indices, group
labels, and right indices of a tuple, and sends $u$
to the global identity.  In particular it preserves
the natural meet with that identity.
Choose nonempty $K$ with $I\hookrightarrow\{1,2\}^K$
and use the diagonal embedding $G\hookrightarrow G^K$.
This embeds $B_I^1(G)$ in the quotient above.
The reverse variety inclusion follows by selecting two
indices of $I$, proving index independence.

For the remaining equality, inside $B_2^1\times
\flat(G)$ let $J$ again consist of elements
with a zero coordinate, and put
\[
 R=J\cup\{(e_{ij},g):i,j\in\{1,2\},\ g\in G\}
       \cup\{(1,1_G)\}.
\]
The same closure argument applies.  Specifically,
$(1,1_G)+(e_{ij},g)$ stays nonzero modulo $J$
exactly when $i=j$ and $g=1_G$.
Thus $R/J\cong B_2^1(G)$.
Conversely $B_2^1$ embeds using identity group labels,
and $\flat(G)$ embeds as a diagonal corner with
zero.  This proves the join formula.
All subsemirings and quotients preserve zero.
\end{proof}

\begin{corollary}\label{cor:brandt-monoid-fb}
If $H$ is a finite
group and $|I|\geq2$, then $B_I^1(H)$ is
nonfinitely based, including infinite index sets.
More generally the same conclusion holds for $B_I^1(G)$
whenever there are group embeddings
$H\leq G\leq H^K$ for a finite group $H$ and
a nonempty set $K$.
\end{corollary}

\begin{proof}
The finite semiring $B_2^1(H)$ contains
$\{1,e_{12},0\}\cong S_7$, and is nonfinitely based
by Lemma~\ref{lem:s7-snfb}.
For the more general assertion the zero-smash lemma
and a diagonal embedding give
$\mathsf{V}(\flat(H^K))=\mathsf{V}(\flat(H))$.
The group embeddings therefore imply
$\mathsf{V}(\flat(G))=\mathsf{V}(\flat(H))$.
Equation~\eqref{eq:brandt-monoid-join} yields
$\mathsf{V}(B_I^1(G))=\mathsf{V}(B_2^1(H))$, so the identical
finite-basis status follows.
\end{proof}

This includes infinite direct powers and infinite
restricted direct products of a finite group $H$.
Taking $H$ nonabelian gives infinite nonabelian examples.
The finite-exponent case was established in
\cite[Corollary 4.12]{JRZ}.

\begin{corollary}[Brandt monoids of finite exponent]
\label{cor:brandt-monoid-finite-exponent}
For any group $G$ of finite exponent and any index
set $I$ with $|I|\geq2$, the natural inverse
semiring $B_I^1(G)$ is nonfinitely based.
\end{corollary}

\begin{proof}
For $I=\{1,2\}$, this is
\cite[Corollary 4.12]{JRZ}, which allows infinite groups.
Theorem~\ref{thm:brandt-monoid-join} gives the assertion
for all $I$.  Lemma~\ref{lem:constant} transfers it
between the two signatures considered here.
\end{proof}

\subsection{Infinite examples}
\label{sec:infinite-natural-boundaries}

\begin{proposition}[An infinite finitely based extension]
\label{prop:s7-infinite-boundary}
There is a countably infinite, locally finite, finitely
based ai-semiring $A$ containing $S_7$ as a
subsemiring.  Its generated variety is locally finite.
\end{proposition}

\begin{proof}
Let $\mathcal U$ be defined by the ai-semiring axioms and
\[
 xy\approx yx,\qquad x^2\approx x^3.
\]
It is finitely based and contains $S_7$.
On $k$ generators, every multiplicative word reduces
to a commutative monomial with each exponent in
$\{0,1,2\}$; there are at most $3^k-1$ nonempty
such monomials.  Distributivity and idempotent addition
then leave only finitely many sums.  Thus $\mathcal U$
is locally finite.
Its countably generated free algebra $F$ is countably
infinite and generates $\mathcal U$; distinct free
generators are separated in a two-element semilattice
with the operation duplicated, which belongs to $\mathcal U$.
Put $A=F\times S_7$.  It still generates $\mathcal U$.
For any free generator $x$, the element $t=x^2$
satisfies $t^2=t$, and $s\mapsto(t,s)$ embeds
$S_7$ into $A$.
For the named-zero signature, add the three absorption
laws to $\mathcal U$ and use $s\mapsto(0_F,s)$.
The same normal form and argument apply.
\end{proof}

We construct a finitely based non-locally-Clifford natural
inverse semiring and a finitely based Clifford example of
bounded exponent containing a nonabelian finite subgroup.

We use the following established representation theorem
of Garvacki\u{\i} \cite{Garv71}, in the explicit formulation
of \cite[Theorem 2]{DT11}.
The algebras representable by partial one-to-one maps,
with composition and intersection, form the variety
$\mathcal R_{\cap}$ defined by the ai-semiring axioms
and the following five identities:
\begin{equation}\label{eq:partial-injection-basis}
\begin{aligned}
 xv+uv+uy&\approx xv+uv+uy+xy,\\
 v+uv+uy&\approx v+uv+uy+y,\\
 xv+v+y&\approx xv+v+y+xy,\\
 x+u+uy&\approx x+u+uy+xy,\\
 xv+uv+u&\approx xv+uv+u+x.
\end{aligned}
\end{equation}
The source writes the first formula with variables
allowed in the formally unitized semigroup, provided
all four displayed products are nonempty.
Its nontrivial specializations are exactly these five
binary-language identities.  Every nontrivial semilattice
with left-zero multiplication satisfies the first identity
and cannot be a semigroup of partial injections.
The meet-preserving Wagner--Preston representation puts
every natural inverse semiring in
$\mathcal R_{\cap}$; see also \cite[Section 7]{JRZ}.

\begin{lemma}[Finite traces]\label{lem:partial-injection-finite-trace}
Every identity failing in an algebra of partial injections
under composition and intersection already fails in the
full symmetric inverse semiring on a finite set.
This also holds with the empty map named as zero.
\end{lemma}

\begin{proof}
Orient a failed identity as $\mathbf{p}\approx\mathbf{q}$ and choose a
pair $(a,b)$ in the value of $\mathbf{p}$ but not that of $\mathbf{q}$.
Distributivity expands each constant-free term into a
nonempty intersection of multiplicative words.
For every word occurring in $\mathbf{p}$, retain its path from
$a$ to $b$ under the given assignment.
Let $Y$ be the finite set of all vertices on these
finitely many finite paths.
Restrict every assigned partial injection to
$Y\times Y$.
All the selected paths remain, so $(a,b)$ still
belongs to the value of $\mathbf{p}$.
Composition and intersection are monotone under
restriction of the assigned relations.
The new value of $\mathbf{q}$ is therefore contained in its
old value and still omits $(a,b)$.
This is a failed evaluation in the symmetric inverse
semiring on $Y$.
If zero is named, any term containing it is the empty
map.  Such a term cannot be the chosen side $\mathbf{p}$;
if it is $\mathbf{q}$, the same argument applies.
\end{proof}

\begin{proposition}\label{prop:infinite-natural-fb-counterexample}
Let $T$ be the semiring of all partial bijections of
$\mathbb N$ having finite domain, with composition
and intersection.  Both $T$ and
\[
 T^1=T\cup\{\operatorname{id}_{\mathbb N}\}
\]
are countably infinite, locally finite natural inverse
semirings whose multiplicative semigroups are periodic
and not locally Clifford.  Both generate
$\mathcal R_{\cap}$, and hence have the explicit finite
basis above.  They contain $B_2^1$ and $S_7$.
All assertions hold in both stated signatures.
\end{proposition}

\begin{proof}
Finite partial bijections are closed under composition,
inversion and intersection.  Adjoining the full identity
preserves these properties.
The natural order is restriction, whose meet is intersection.
The sets are countably infinite.
For finitely many elements of $T$, the union $Y$ of
their domains and ranges is finite; they generate a
subsemiring of the finite symmetric inverse semiring
on $Y$.
For $T^1$, the generated subsemiring is contained in
that finite semiring together with the full identity.
Thus both algebras are locally finite, and in particular
every multiplicative cyclic subsemigroup is finite.

On a fixed two-point subset, take its identity, the four
singleton partial bijections, and the empty map.
They form $B_2^1$.  Its identity $e$, one off-diagonal
arrow $a$, and zero form $S_7$.
The arrow lies in $eTe$ and has distinct domain and
range idempotents, proving failure of local Cliffordness.

Every finite symmetric inverse semiring embeds in $T$.
The finite-trace lemma and the representation theorem
therefore give $\mathsf{V}(T)=\mathcal R_{\cap}$.
Since $T\subseteq T^1$ and $T^1\in\mathcal R_{\cap}$,
the latter has the same variety.
For the named-zero signature, add the three absorption
laws to the displayed finite basis.
To justify this directly, represent a model of those laws
as partial injections.  Its zero is an idempotent map,
thus the identity on some set $D$, and is contained
in every represented map.  All maps fix $D$, and
injectivity prevents points outside $D$ from mapping
into $D$.
Restricting to the complement of $D$ preserves both
operations and injectivity of the representation, and
sends zero to the empty map.
The finite-trace argument then applies unchanged.
\end{proof}

The variety in this proposition is not locally finite.
For example $\flat(\mathbb{Z})$ is represented by the
integer translations and the empty map, and is an infinite
two-generated member of $\mathcal R_{\cap}$.
For comparison, the orthogonal sum of countably many
copies of $B_2^1$ is also countably infinite, locally
finite, periodic and non-locally-Clifford, but is nonfinitely
based.  It embeds on the coordinate axes of a direct power
of $B_2^1$, and contains a copy of $B_2^1$, so its
variety is exactly $\mathsf{V}(B_2^1)$.

\begin{proposition}\label{prop:infinite-clifford-sylow-boundary}
There is a countably infinite locally finite group $G$
of exponent $3$ such that $\flat(G)$ is finitely
based and generates a locally finite variety, although
$G$ contains a nonabelian finite $3$-group.
\end{proposition}

\begin{proof}
Let $\mathcal N$ be the group variety defined by
\[
 x^3\approx1,\qquad [[x,y],z]\approx1,
\]
where $[x,y]=x^{-1}y^{-1}xy$.
It is locally finite by an explicit collecting argument:
in a group generated by $g_1,\ldots,g_d$, centrality
of commutators puts every element in the form
\[
 g_1^{a_1}\cdots g_d^{a_d}
 \prod_{i<j}[g_i,g_j]^{b_{ij}},
 \qquad a_i,b_{ij}\in\{0,1,2\}.
\]
There are at most $3^{d+\binom d2}$ such forms.

Let $H_1,H_2,\ldots$ represent the isomorphism types of
finite groups in $\mathcal N$, and set
\[
 G=\bigoplus_{i\geq1}H_i,
\]
the restricted direct product with finite coordinate
support.  This is a countably infinite member of
$\mathcal N$, is locally finite, and contains every
finite group in $\mathcal N$, up to isomorphism, as a
coordinate subgroup.
Every quasi-identity failing in a group of $\mathcal N$
fails on a finitely generated, hence finite, subgroup.
That subgroup embeds in $G$.
Consequently the quasi-identities of $G$ are exactly
those of $\mathcal N$.
They have the displayed finite basis relative to groups.
The finite-exponent correspondence
\cite[Lemma 2.2]{RJZL} proves that $\flat(G)$
is finitely based.

Its generated semiring variety is also
locally finite directly.
A multiplicative word on $d$ variables is determined,
as a word function on $\flat(G)$, by its variable
content and its value in the relatively free group of
$\mathcal N$ on those variables.
The collecting bound makes the number of these functions
finite.  Every semiring term is an idempotent sum of them,
so there are only finitely many term functions in each
fixed finite number of variables.

Finally $H=\operatorname{UT}_3(\mathbb F_3)$ is a
nonabelian group of order $27$, of class two and
exponent three: $(I+N)^3=I$ for every strictly upper
triangular $3$-by-$3$ matrix $N$.
Thus $H$ embeds in $G$, and $\flat(H)$ is a
subsemiring of $\flat(G)$.
The finite criterion \eqref{eq:group-criterion} makes
$\flat(H)$ nonfinitely based.
The construction and Lemma~\ref{lem:constant} also give
the named-zero assertion.
\end{proof}

\subsection{Finite generating algebras}

\begin{theorem}\label{thm:finite-natural-generator-replacement}
Let $S$ be a natural inverse semiring.
If $\mathsf{V}(S)=\mathsf{V}(B)$ for some finite ai-semiring $B$,
there is a finite natural inverse semiring
$A$, embedded in a finite direct power of $S$, such
that $\mathsf{V}(A)=\mathsf{V}(S)$.
In particular, if such an $S$ is not locally Clifford,
it is nonfinitely based and has no bounded-variable
identity basis.  These assertions hold in both signatures.
\end{theorem}

\begin{proof}
The trivial variety is immediate, so assume it is nontrivial.
Since $B\in\mathsf{V}(S)$, choose a subsemiring $D\leq S^I$
and an onto homomorphism $q:D\to B$.
Choose one preimage for each element of $B$, and let
$R$ be their generated subsemiring in $D$.
The variety of the finite algebra $B$ is locally finite,
so $R$ is finite and $q|_R$ is still onto.
Inside $S^I$, let $A_0$ be generated by
$R\cup\{r^{-1}:r\in R\}$.
It is finite for the same local-finiteness reason.
Inversion reverses products and preserves natural meets,
so induction on terms shows that $A_0$ is inverse-closed.
It is also meet-closed, and its natural order is the
restriction of the ambient natural order.
Thus $A_0$ is a finite natural inverse semiring.

Since $R\leq A_0$ and $q|_R:R\to B$ is onto,
$B\in\mathsf{V}(A_0)$.  Hence
\[
 \mathsf{V}(B)\subseteq\mathsf{V}(A_0)\subseteq\mathsf{V}(S)=\mathsf{V}(B).
\]
Finitely many coordinates separate the finitely many
distinct pairs of elements of $A_0$.
Projection onto those coordinates is an embedding and
gives the required finite-power copy $A$.

If $S$ is not locally Clifford, choose an idempotent
$e$ and $a\in eSe$ with $d(a)\ne r(a)$.
The subsemiring generated by $e,a,a^{-1}$ is finite,
and the inverses of all its generators are present.
Since inversion reverses products and preserves natural
meets, induction on terms makes it inverse-closed.
It is still not locally Clifford.
Lemma~\ref{lem:brandt-monoid-divisor} puts $B_2^1$,
and hence $S_7$, in its generated variety.
Therefore $S_7\in\mathsf{V}(S)=\mathsf{V}(B)$.
Lemma~\ref{lem:s7-snfb} makes $B$, and hence $S$,
nonfinitely based.
Lemma~\ref{lem:locally-finite-rank} excludes a
bounded-variable basis.
In the named-zero language all generated subsemirings
and maps above preserve the designated zero.
\end{proof}

\begin{theorem}\label{thm:finite-variety-natural-classification}
Let $S$ be a natural inverse semiring such that
$\mathsf{V}(S)=\mathsf{V}(B)$ for some finite ai-semiring $B$.
Then $S$ is finitely based if and only if it is locally
Clifford and at least one of the following holds:
\begin{enumerate}[$(\rm i)$]
\item $S$ is combinatorial;
\item $S$ is Clifford, and every finite $p$-subgroup
of every maximal subgroup of $S$ is abelian, for every
prime $p$.
\end{enumerate}
Every nonfinitely based case has no bounded-variable identity
basis.  The criterion holds in both signatures, without a
cardinality restriction on $S$.
\end{theorem}

\begin{proof}
By Theorem~\ref{thm:finite-natural-generator-replacement},
there are a finite natural inverse semiring $A$ and an
embedding $A\leq S^n$, for some finite $n$, such that
$\mathsf{V}(A)=\mathsf{V}(S)$.  Finite basability is therefore the same
for $A$ and $S$.

We first compare their structural properties.
Local Cliffordness is equivalent to
\eqref{eq:regular-flat-identity} in the natural inverse
class, by Theorem~\ref{thm:regular-flat-identity}.
Since $A,S$ have the same identities, one is locally
Clifford exactly when the other is.
Cliffordness and being combinatorial pass from $S$ to
the inverse subsemigroup $A\leq S^n$: domain and range
idempotents agree coordinatewise in the Clifford case,
and all coordinate images of a subgroup are trivial in
the combinatorial case.
They also pass from $A$ to $S$, using the finiteness
of $A$.
If $A$ is Clifford, a common finite exponent $r$ of
its maximal subgroups gives the identity
$x^{r+1}\approx x$.  It holds in $S$, forcing every
element of $S$ to lie in a cyclic subgroup.
If $A$ is combinatorial, its finite multiplicative
semigroup satisfies $x^{m+1}\approx x^m$ for some
$m\geq1$.  This identity holds in $S$ and excludes
every nontrivial subgroup.  Thus each of the three
structural properties agrees for $A$ and $S$.

Suppose first that $S$ satisfies the proposed criterion.
If it is combinatorial, so is $A$; the finite
classification makes $A$, and hence $S$, finitely
based.
Otherwise $S$ and $A$ are Clifford.
Let $P$ be a Sylow subgroup of any maximal subgroup of
$A$.  It is a finite $p$-group for some prime $p$.
Each coordinate projection of the embedding $A\leq S^n$
maps $P$ to a finite $p$-subgroup of a maximal
subgroup of $S$.  These images are abelian by hypothesis.
The projections jointly separate elements, so any two
elements of $P$ commute.  All the required Sylow
subgroups of $A$ are abelian, and the finite
classification again gives finite basability.

Conversely suppose $S$, and therefore $A$, is finitely
based.  The finite classification makes $A$ locally
Clifford, so $S$ is locally Clifford as well.
If $A$ is combinatorial, the structural comparison
already proves (i).
Otherwise $A$, and hence $S$, is Clifford.
Suppose a maximal subgroup $G_e$ of $S$ had a finite
nonabelian $p$-subgroup $P$.
The flat-group projection of
Theorem~\ref{thm:clifford-natural} gives
\[
 \flat(P)\in\mathsf{V}(\flat(G_e))
 \subseteq\mathsf{V}(S)=\mathsf{V}(A).
\]
If $S$ has no zero, first use the identity-preserving
binary-language zero adjunction of
Remark~\ref{rem:clifford-nozero}.
The finite-target theorem
\cite[Theorem 6.1]{JRZ}, already used in
Theorem~\ref{thm:three-obstructions}, now makes the finite
ai-semiring $A$ nonfinitely based: its variety contains
the finite semiring $\flat(P)$, which has a
nonabelian nilpotent subgroup.
This contradicts the choice of $A$.  Thus (ii) holds.

Finally $\mathsf{V}(S)=\mathsf{V}(B)$ is locally finite because $B$
is finite and the signature is finite.
Lemma~\ref{lem:locally-finite-rank} therefore excludes a
bounded-variable identity basis in every negative case.
\end{proof}

\begin{remark}\label{rem:finite-variety-group-period}
Under the hypothesis of
Theorem~\ref{thm:finite-variety-natural-classification},
every maximal subgroup of $S$ is locally finite, and
these groups have a common finite exponent.
For the first assertion, a finite set of group generators
together with their inverses generates a finite subsemiring,
by local finiteness of $\mathsf{V}(B)$; the generated group is
contained in that finite subsemiring.
For the second, the finite multiplicative semigroup of
$B$ satisfies $x^{m+r}\approx x^m$ for some
$m,r\geq1$.  The same identity holds in $S$.
Inside any subgroup it cancels to $g^r=1$.
The classification applies when $\mathsf{V}(S)$ has a finite
generating algebra.
\end{remark}

\section{Inherent nonfinite basability}
\label{sec:inherent}

\subsection{Locally finite envelopes}
\label{sec:not-inherent}

For an algebra generating a locally finite variety,
\emph{inherently nonfinitely based} means that no locally
finite finitely based variety of its signature contains
it.  Finite regular flat semirings and finite naturally
ordered Clifford semirings admit explicit locally finite
finitely based envelopes.
The flat part also follows from the more general
\cite[Proposition 2.22]{RYZ26}, which excludes inherent
nonfinite basability for every finite ai-semiring of
additive height at most two.  The proof below constructs
a locally finite finitely based envelope and also applies
to finite natural Clifford semirings.

\begin{proposition}\label{prop:not-inherent}
Every finite regular flat semiring, and every finite
natural Clifford semiring, belongs to a locally
finite finitely based variety of ai-semirings.  Therefore
none is inherently nonfinitely based, in either stated
language.
\end{proposition}

\begin{proof}
Let $S$ be nonzero.  In the regular flat case use
its decomposition and finite product group $G$.
In the natural Clifford case use the finite product group
in Theorem~\ref{thm:clifford-natural}.
The exact variety formulas and the group-corner inclusion
give $S\in\mathsf{V}(B_2(G))$ in both cases.
Thus its multiplicative reduct
belongs to the semigroup variety of the finite
multiplicative semigroup $T=B_2(G)$.
The latter has a finite identity basis $\Delta$, by
\cite[Corollary 2]{Volkov19}; that result includes the
finite-group case using the finite group basis theorem.
For the trivial group one may use the usual three
Brandt semigroup identities.

Add $\Delta$ to the finite ai-semiring axioms and call
the resulting variety $\mathcal U$.  It contains $S$
and is finitely based.  It is also locally finite.
Indeed, the semigroup variety generated by finite $T$
is locally finite: on $k$ generators its free elements
are term functions on the finite set $T^k$, of which
there are finitely many.  Fix representatives of these
word classes.  Distributivity expands every semiring term
in $k$ variables into a sum of words, and $\Delta$
reduces each word to one of the fixed representatives.
Commutative idempotent addition leaves finitely many
possible sums.  Thus the free $k$-generated semiring
in $\mathcal U$ is finite.

In the named-zero language add the three absorption laws;
terms containing zero give one additional normal form.
The same local finiteness proof applies.  The zero-only
algebra has a trivial envelope separately.
\end{proof}

The proposition supplies explicit locally finite finitely
based envelopes for the examples in
Theorems~\ref{thm:seven} and~\ref{thm:cyclic-minimum}.

\subsection{Zimin words and the structural classification}

The binary-language inherent nonfinite basability of
$B_2^1$ was established in \cite[Corollary 2.25]{RYZ26}.
We give a direct proof of Zimin minimality and apply
\cite[Theorem 1.1]{RYZ26} to obtain the named-zero result
and the structural classification.

Put $Z_1=x_1$ and $Z_{n+1}=Z_nx_{n+1}Z_n$.
A nonempty word $\mathbf{w}$ is \emph{minimal in the additive
order} of an ai-semiring if
$\mathbf{u}+\mathbf{w}\approx \mathbf{w}$ implies literal equality $\mathbf{u}=\mathbf{w}$
for every nonempty word $\mathbf{u}$.
The additive order is opposite to the natural inverse
order when addition is natural meet.

\begin{lemma}\label{lem:brandt-zimin-minimal}
Every Zimin word is minimal in the additive order of
the naturally semilattice-ordered Brandt monoid $B_2^1$.
\end{lemma}

\begin{proof}
Write $c=e_{12}$, $d=e_{21}$, and suppose
\[
 B_2^1\models \mathbf{u}+Z_n\approx Z_n.
\]
Equivalently $Z_n\leq_{\rm nat}\mathbf{u}$ under every assignment.
An extra variable in $\mathbf{u}$ is impossible: assign it
zero and all other variables the identity, obtaining
the false inequality $1\leq_{\rm nat}0$.
Thus $\mathbf{u}$ uses only $x_1,\ldots,x_n$.

For each $i$, assign
\[
 x_j\longmapsto
 \begin{cases}
 1,&j<i,\\
 c,&j=i,\\
 d,&j>i.
 \end{cases}
\]
The recursion and $cdc=c$ give value $c$ for
$Z_n$.
The off-diagonal matrix unit $c$ is maximal in the
natural order and is incomparable with $1$;
therefore $\mathbf{u}$ also has value $c$.

Let $\mathbf{u}_i$ be the word obtained from $\mathbf{u}$ by deleting
all letters of index less than $i$, with the auxiliary
convention $\mathbf{u}_{n+1}=\varepsilon$.
Under this assignment its $x_i$'s read as $c$ and
its larger letters as $d$.
The relations $c^2=d^2=0$ force any nonzero such
word to alternate.
An alternating word has value $c$ exactly when it
begins and ends with $c$; even lengths give $cd$
or $dc$, and odd lengths beginning with $d$ give
$d$.
Consequently
\[
 \mathbf{u}_i=x_i y_1x_i y_2x_i\cdots y_qx_i,\qquad
 \mathbf{u}_{i+1}=y_1\cdots y_q,
\]
where $q\geq0$ and each $y_j$ has index greater than $i$.
For $i=n$ this gives $\mathbf{u}_n=x_n$.
Descending induction then reconstructs each $\mathbf{u}_i$
uniquely by inserting $x_i$ before, between, and
after the letters of $\mathbf{u}_{i+1}$.
These are exactly the successive projections of
$Z_n$.  Hence $\mathbf{u}=\mathbf{u}_1=Z_n$ literally.
\end{proof}

By \cite[Theorem 1.1 and its construction]{RYZ26},
if every Zimin word is minimal for an ai-semiring $A$,
then for every $k\geq1$ there is an infinite finitely
generated flat ai-semiring $T_k$ satisfying all its
identities in at most $k$ variables.
In particular a finite $A$ with this property is
inherently nonfinitely based.

\begin{corollary}\label{cor:brandt-monoid-inherent}
The naturally semilattice-ordered Brandt monoid $B_2^1$ is
inherently nonfinitely based in both stated signatures.
\end{corollary}

\begin{proof}
In the binary language apply the quoted criterion to
Lemma~\ref{lem:brandt-zimin-minimal}.
For the named-zero language, a valid identity of
$B_2^1$ either contains no zero constant or contains
zero on both sides.
Indeed any term containing zero is identically zero,
whereas a constant-free term takes value $1$ when
all variables are $1$.
Name the existing absorbing zero of each $T_k$.
It now satisfies all at-most-$k$-variable identities
of the expansion, including those with zero.
If a locally finite finitely based variety contained
that expansion, choose $k$ bounding its finite basis.
Then it would contain the infinite finitely generated
$T_k$, a contradiction.
\end{proof}

\begin{theorem}\label{thm:finite-natural-inherent}
For a finite natural inverse semiring $S$,
in either stated signature,
\[
 S\text{ is inherently nonfinitely based}
 \quad\Longleftrightarrow\quad
 S\text{ is not locally Clifford}.
\]
\end{theorem}

\begin{proof}
If $S$ is locally Clifford, its finite regular flat
orthogonalization has the same identities by
Theorem~\ref{thm:local-clifford-representation}.
Proposition~\ref{prop:not-inherent} puts that
orthogonalization, and hence $S$, in a locally finite
finitely based variety.  The zero-only case has a trivial
envelope as well.
If $S$ is not locally Clifford, the semiring divisor
lemma gives $B_2^1\in\mathsf{V}(S)$.
By Corollary~\ref{cor:brandt-monoid-inherent}, no locally
finite finitely based variety can contain $S$, since
it would then contain $B_2^1$.
This is precisely the required INFB assertion; finiteness
ensures that $\mathsf{V}(S)$ itself is locally finite.
\end{proof}

\begin{corollary}\label{cor:locally-finite-variety-nonlocal}
Let $S$ be a natural inverse semiring with
$\mathsf{V}(S)$ locally finite.  If $S$ is not locally
Clifford, it is inherently nonfinitely based and has no
bounded-variable identity basis.
If, more specifically, $\mathsf{V}(S)=\mathsf{V}(B)$ for some finite
algebra $B$, then
\[
 S\text{ is inherently nonfinitely based}
 \quad\Longleftrightarrow\quad
 S\text{ is not locally Clifford}.
\]
There is no cardinality restriction on $S$.
Both assertions hold in both signatures.
\end{corollary}

\begin{proof}
In the nonlocal case choose an idempotent $e$ and
$a\in eSe$ with $d(a)\ne r(a)$.
The subsemiring $N=\langle e,a,a^{-1}\rangle$ is finite
by local finiteness of $\mathsf{V}(S)$.
The inverses of its generators are present, inversion
reverses products, and inversion preserves natural meets.
Induction on terms therefore makes $N$ inverse-closed.
It remains non-locally-Clifford, so the finite divisor
lemma gives $B_2^1\in\mathsf{V}(N)\subseteq\mathsf{V}(S)$.
No locally finite finitely based variety containing $S$
can then exist, by Corollary~\ref{cor:brandt-monoid-inherent}.
The bounded-variable assertion follows from
Lemma~\ref{lem:locally-finite-rank}.

Under the finite-generator hypothesis, take the finite
natural inverse replacement $A\leq S^n$ from
Theorem~\ref{thm:finite-natural-generator-replacement}.
If $S$ is locally Clifford, so is this inverse
subsemigroup $A$ of its direct power.
Theorem~\ref{thm:finite-natural-inherent} puts $A$
in a locally finite finitely based variety.
Since $\mathsf{V}(A)=\mathsf{V}(S)$, the same variety contains $S$.
This proves the reverse direction.
\end{proof}

The semiring $S_7$ is strongly nonfinitely based and
belongs to a locally finite finitely based variety by
Proposition~\ref{prop:s7-infinite-boundary}.


\begin{thebibliography}{99}
\bibitem{GV}
S.~V.~Gusev and M.~V.~Volkov,
\emph{Semiring identities of finite inverse semigroups},
Semigroup Forum \textbf{106} (2023), 403--420.
\href{https://arxiv.org/abs/2204.10514}{arXiv:2204.10514}.
\bibitem{Sh}
V.~Yu.~Shaprynski\v{\i},
\emph{Semiring identities in the semigroup $B_2$},
preprint (2026).
\href{https://arxiv.org/abs/2604.01706}{arXiv:2604.01706}.
\bibitem{JRZ}
M.~Jackson, M.~Ren and X.~Zhao,
\emph{Nonfinitely based ai-semirings with finitely based semigroup
reducts}, Journal of Algebra \textbf{611} (2022), 211--245.
\href{https://arxiv.org/abs/2112.13918}{arXiv:2112.13918}.
\bibitem{RJZL}
M.~Ren, M.~Jackson, X.~Zhao and D.~Lei,
\emph{Flat extensions of groups and limit varieties of additively
idempotent semirings}, Journal of Algebra \textbf{623} (2023), 64--85.
\href{https://arxiv.org/abs/2208.13973}{arXiv:2208.13973}.
\bibitem{Jackson08}
M.~Jackson,
\emph{Flat algebras and the translation of universal Horn
logic to equational logic}, Journal of Symbolic Logic
\textbf{73} (2008), 90--128.
\href{https://doi.org/10.2178/jsl/1208358744}{doi:10.2178/jsl/1208358744}.
\bibitem{Jackson19}
M.~Jackson,
\emph{Low growth equational complexity}, Proceedings of the
Edinburgh Mathematical Society \textbf{62} (2019), 197--210.
\href{https://arxiv.org/abs/1607.07156}{arXiv:1607.07156}.
\bibitem{MS25}
G.~Metcalfe and S.~Santschi,
\emph{Equational theories of idempotent semifields},
Bulletin of the London Mathematical Society
\textbf{57}(3) (2025), 771--785.
\href{https://doi.org/10.1112/blms.13228}{doi:10.1112/blms.13228}.
\href{https://arxiv.org/abs/2402.09876}{arXiv:2402.09876}.
\bibitem{GRZ26}
Z.~Gao, M.~Ren and Y.~Zhou,
\emph{The finite basis problem for flat semirings
$S_c(W)$, $M_c(W)$ and $M(W)$}, preprint (2026).
\href{https://arxiv.org/abs/2609.40155v1}{arXiv:2609.40155v1}.
\bibitem{Garv71}
V.~S.~Garvacki\u{\i},
\emph{$\cap$-semigroups of transformations},
Teoriya Polugrupp i ee Prilozheniya \textbf{2} (1971),
3--13 (in Russian).
\bibitem{DT11}
W.~A.~Dudek and V.~S.~Trokhimenko,
\emph{The relation of semiadjacency on
$\cap$-semigroups of transformations},
Semigroup Forum \textbf{82} (2011), 109--120.
\href{https://doi.org/10.1007/s00233-010-9269-3}{doi:10.1007/s00233-010-9269-3}.
\bibitem{Volkov19}
M.~V.~Volkov,
\emph{Identities in Brandt semigroups, revisited},
Ural Mathematical Journal \textbf{5}(2) (2019), 80--93.
\href{https://doi.org/10.15826/umj.2019.2.008}{doi:10.15826/umj.2019.2.008}.
\href{https://arxiv.org/abs/1906.08175}{arXiv:1906.08175}.
\bibitem{RYZ26}
M.~Ren, M.~Yue and Y.~Zhou,
\emph{Inherently nonfinitely based additively idempotent
semirings}, preprint (2026), version 3.
\href{https://arxiv.org/abs/2609.27641v3}{arXiv:2609.27641v3}.
\end{thebibliography}
\end{document}